\documentclass[11pt]{amsart}
\usepackage{amsmath}
\usepackage{amssymb}

\theoremstyle{plain}
\newtheorem{theorem}{\bf Theorem}[section]
\newtheorem{proposition}[theorem]{\bf Proposition}
\newtheorem{lemma}[theorem]{\bf Lemma}

\newtheorem{problem}[theorem]{\bf Problem}

\theoremstyle{definition}
\newtheorem{example}[theorem]{\bf Example}

\newtheorem{definition}[theorem]{\bf Definition}
\newtheorem{remark}[theorem]{\bf Remark}

\newcommand{\N}{\mathbb N}
\newcommand{\Z}{\mathbb Z}

 \DeclareMathOperator{\ind}{ind}
 \DeclareMathOperator{\ord}{ord}

\renewcommand{\time}{\negthinspace \times \negthinspace}

\numberwithin{equation}{section}

\begin{document}
\title[Minimal zero-sum sequences of length four]{Minimal zero-sum sequences of length four over finite cyclic groups II}

\author{Yuanlin Li* and Jiangtao Peng}
\address{Department of Mathematics, Brock University, St. Catharines, Ontario,
Canada L2S 3A1} \email{yli@brocku.ca}

\address{College of Science, Civil Aviation University of China, Tianjin 300300, P.R. China} \email{jtpeng1982@yahoo.com.cn}

\thanks{This research was supported in part by a Discovery
Grant from the Natural Sciences and Engineering Research Council of
Canada,  the National Natural Science Foundation of China (No. 11126137 and No. 11271250) and a research grant from  Civil Aviation University of China.\\
*Corresponding author: Yuanlin Li, Department of Mathematics, Brock
University, St. Catharines, Ontario Canada L2S 3A1, Fax:(905)
378-5713;\\ E-mail: yli@brocku.ca (Y. Li) \\ \today}

\begin{abstract}
Let $G$ be a finite cyclic group. Every sequence $S$ over $G$ can be written in the form
$S=(n_1g)\cdot\ldots\cdot(n_lg)$ where $g\in G$ and $n_1, \ldots, n_l\in[1, \ord(g)]$, and the index $\ind(S)$ of $S$ is defined to be the minimum of $(n_1+\cdots+n_l)/\ord(g)$ over all possible $g\in G$ such that $\langle g \rangle =G$.
 An open problem on the index of length four sequences asks whether or not every minimal zero-sum sequence of length 4 over a finite cyclic group $G$ with $\gcd(|G|, 6)=1$ has index 1. In this paper, we show that if $G=\langle g\rangle$ is a cyclic group with order of a product of two prime powers and $\gcd(|G|, 6)=1$, then every minimal zero-sum sequence $S$ of the form $S=(g)(n_2g)(n_3g)(n_4g)$ has index 1. In particular, our result confirms that the above problem has an affirmative answer when the order of $G$ is a product of two different prime numbers or a prime power, extending a recent result by the first author, Plyley, Yuan and Zeng.
\end{abstract}
\maketitle

\section{Introduction}
\bigskip

Throughout this paper, let $G$ be an additively written finite cyclic
group of order $|G|=n$.  By a sequence over $G$ we mean a finite sequence of terms from $G$ which is unordered and repetition of terms is allowed. We view sequences over $G$ as elements of the free abelian moniod   $\mathcal{F}(G)$ and use multiplication notation. Thus a sequence  $S$ of length
$|S|= k$ is written in the
form $S=(n_1g)\cdot \, \ldots \,\cdot (n_kg)$ where $n_1,..., n_k \in \N $ and  $g\in G$. We call $S$  a \emph{zero-sum
sequence} if the sum of $S$ is zero (i.e. $\sum_{ i=1}^k n_ig=0$). If $S$ is a zero-sum sequence, but no proper
nontrivial subsequence of $S$ has sum zero, then  $S$ is called a \emph{minimal zero-sum sequence}.  Recall that the index of a sequence $S$ over $G$ is defined as follows.

\begin{definition}
 For a sequence over $G$
 $$S=(n_1g)\cdot\ldots\cdot(n_lg), \,\,\, \mbox{where} \,\, 1\le n_1, \ldots, n_l \le n,$$
 the index of $S$ is defined by
 $\ind(S)=\min\{\| S \|_g \,|\,g\in G \, \mbox{with}\,\, G=\langle g \rangle\}$ where  $$\|S\|_g=\frac{n_1+\cdots+n_l}{\ord(g)}.$$
\end{definition}
\noindent Clearly, $S$ has sum zero if and only if $\ind(S)$ is an integer. There are also slightly different definitions of the index in the literature, but they are all equivalent (see Lemma 5.1.2 in \cite{Ge:09a}).

 The index of a sequence is a crucial invariant
in the investigation of (minimal) zero-sum sequences (resp. of
zero-sum free sequences) over cyclic groups. It was first addressed
by Kleitman-Lemke (in the conjecture \cite[page 344]{KL:89}),
used as a key tool by Geroldinger (\cite[page 736]{G:87}), and
then investigated by Gao \cite{Gao:00} in a systematical way. Since then it has received a great deal of attention (see for example \cite{CFS:99, CS:05, GG:09, Ge:09a, LPYZ:10,
P:04, SC:07, XY:09, Y:09}).

 A main focus of the investigation of index is to determine minimal zero-sum sequences of index 1. If $S$ is a minimal zero-sum sequence of length $|S|$ such that $|S|\leq 3$ or $|S| \geq \lfloor\frac{n}{2}\rfloor+2$, then
ind(S) = 1 (see \cite{CFS:99,SC:07,Y:09}). In contrast to that, it was shown that for each $k \mbox{ with }  5 \leq k \le \lfloor\frac{n}{2}\rfloor + 1$,
there is a minimal zero-sum sequence $T$ of length $|T| = k$ with $\ind(T)\geq 2$ (\cite{SC:07, Y:09})  and that the
same is true for $k = 4$ and $\gcd(n, 6) \ne 1$ (\cite{P:04}). The only unsolved case is that whether or not every minimal zero-sum sequence of length 4 in a cyclic group $G$
with $\gcd(|G|, 6)=1$ has index 1 and  this leads to the following open problem.

\bigskip

\begin{problem}\label{problem}
Let $G$ be a finite cyclic group such that $\gcd(|G|, 6)=1$. Is it true that every minimal zero-sum sequence $S$ over $G$ of length $|S|=4$ has $\ind(S)=1$?
\end{problem}

In a recent paper \cite{LPYZ:10} the first author together with Plyley, Yuan, and Zeng  proved that the open problem (Problem~\ref{problem}) has an affirmative answer if $n$ is a prime power. However,  the general case is
still open. In this paper, we attempt to answer this problem affirmatively for a more general case when $n$ is a product of two prime powers. Our main result is as follows.

\begin{theorem}\label{maintheorem1}
Let $n=p_1^{\alpha} \cdot p_2 ^{\beta} $, where $p_1 \not = p_2 $ are primes and $ \alpha, \beta \in \N$, and  $\gcd(n,6)=1$. Let $S=(g)(x_2g)(x_3g)(x_4g)$ be a minimal zero-sum sequence over $G$ such that $\ord(g)=n$ and $1 \le x_2, x_3, x_4 \le n-1$.  Then $\ind(S)=1$.
\end{theorem}

\bigskip
As applications of Theorem \ref{maintheorem1}, we obtain that the problem has an affirmative answer for the group with order of a product of two primes (Theorem \ref{thm-pq}) as well as for the group of prime power order (Theorem~\ref{theorem for prime power}-- the main result of \cite{LPYZ:10}).

\bigskip

The paper is organized as follows. In next section, we provide some preliminary results, and then a proof of
our main result is given in Section 3. In the last section, we give some applications of  Theorem \ref{maintheorem1}.
\vskip 1 cm

\bigskip

\section{Preliminaries}
\bigskip

 We first list some useful facts and simple results in the following remark. We denote by $|x|_n$ the least positive residue of $x$ modulo $n$, where $n \in \N$ and $x\in \Z$.

\begin{remark}\label{remark}
Let $S = (x_1 g)(x_2g)(x_3g)(x_4g)$  be a minimal zero-sum sequence over  $G$ such that $\mbox{ord}(g)=n =|G|$, and $\ 1 \le x_1 \le x_2 \le x_3 \le x_4 \le n-1$. Then $x_1+x_2+x_3+x_4 =  \nu n,$ where $1 \le \nu \le 3$.

\begin{enumerate}

\item[(0).] To show that $\ind(S)=1$, it suffices to find an integer $m$ with $\gcd(m,n)=1$ such that $|mx_1|_n + |mx_2|_n +|mx_3|_n + |mx_4|_n =n$, and this fact will be frequently used later.  Furthermore, we may always assume that $\nu \ge 2$.

\item[(1).]  It was mentioned in \cite{P:04}  that Problem \ref{problem} was confirmed computationally to hold true if $n \le 1000$ (The claim has been double checked by the second author and Wang by using a computer program). Hence, throughout the paper, we always assume that $n>1000$.

\smallskip
\item[(2).]  If $x_1+x_2+x_3+x_4=3n$, then $|(n-1)x_1|_n + |(n-1)x_2|_n +|(n-1)x_3|_n + |(n-1)x_4|_n = (n-x_1) + (n-x_2) + (n-x_3) + (n-x_4) =n$.  Since $\gcd(n, n-1)=1$, we have $\ind(S)=1$. Thus  we may always assume that   $x_1+x_2+x_3+x_4=2n$.

\smallskip
\item[(3).]  If $x_1 \le x_2 \le x_3 < \frac{n}{2}$, then $x_4= 2n - (x_1+x_2+x_3) > \frac{n}{2}$. Now $|(n-2)x_1|_n + |(n-2)x_1|_n + |(n-2)x_1|_n + |(n-2)x_1|_n =(n- 2x_1) + (n- 2x_2) + (n- 2x_3) + (2n - 2x_4) = n$. Since $\gcd(n, n-2)=1$, we have $\ind(S)=1$.

\smallskip
\item[(4).]  If $x_4 \ge x_3 \ge x_2 > \frac{n}{2}$, then  $x_1< \frac{n}{2}$. Since  $|2x_1|_n + |2x_1|_n + |2x_1|_n + |2x_1|_n = 2x_1 + (2x_2 -n) + (2x_3-n) + (2x_4 -n) = n$ and  $\gcd(n, 2)=1$, we have $\ind(S)=1$.
\end{enumerate}
\end{remark}

\bigskip

Let $S$ be the sequence as described in Theorem \ref{maintheorem1}. By the above remark, we may always assume that $1+x_2+x_3+x_4=2n$ and $ 1<x_2 < n/2 < x_3 \le x_4 < n-1 $. Now let $c=x_2, b=n-x_3, a=n-x_4$, and  it is not hard to show that the following proposition implies Theorem ~ \ref{maintheorem1}.

\begin{proposition}\label{maintheorem}
Let $n=p_1^{\alpha} \cdot p_2 ^{\beta} $, where $p_1 \not = p_2 $ are primes and $ \alpha, \beta \in \N$, and  $\gcd(n,6)=1$. Let $S=(g)(cg)((n-b)g)((n-a)g)$ be a minimal zero-sum sequence over $G$ with $\ord(g)=n$, $1+c=a+b$ and $1 < a \le b < c < \frac{n}{2}.$ Then $\ind(S)=1$.
\end{proposition}

For  any real numbers $a < b \in \mathbb R$, we set $[a, b] = \{ x \in \mathbb Z \mid a \le x \le b\}$ the set of all integers between $a$ and $b$, and similarly, set $[a, b) = \{ x \in \mathbb Z \mid a \le x < b\}$. From now  (until the end of the next section)  we always assume that  $S$ is the sequence as described  in Proposition \ref{maintheorem}. Next we  give a  crucial  lemma.

\medskip
\begin{lemma}\label{mainlemma}
Proposition \ref{maintheorem} holds if one of the following conditions holds\,{\rm :}
\begin{enumerate}
\item[(1).] There exist positive integers $k, m$ such that $\frac{kn}{c} \le m \le\frac{kn}{b}, \ \gcd(m,n)=1, \ 1 \le k \le b,$ and $ma < n$.

\smallskip
\item[(2).] There exists a positive integer $M \in \left [1, \frac{n}{2}\right]$ such that $ \gcd(M,n)=1$ and at least two of the following inequalities hold\,{\rm :}

      $|Ma|_n > \frac{n}{2}, \   |Mb|_n > \frac{n}{2}, \ |Mc|_n < \frac{n}{2}$.
\end{enumerate}
\end{lemma}

\begin{proof}
(1). If $ \gcd(m,n)=1$ and $ 1 \le k \le b,$ we conclude that $\frac{kn}{c} < m < \frac{kn}{b}.$ Since $|m|_n + |mc|_n + |m(n-b)|_n + |m(n-a)|_n \le m + (mc-kn) + (kn-mb) + (n-ma) = n$, we have $\ind(S)=1$.

(2). It follows that at least three elements of $\{  |M|_n,  \  |Mc|_n ,\ |M(n-b)|_n ,\ |M(n-a)|_n \}$ are less than $\frac{n}{2}$. By Remark \ref{remark} (3), we have  $\ind(S)=1$.
\end{proof}

\medskip
As a consequence of Lemma \ref{mainlemma}, we have the following easy  observation.

\medskip
\begin{lemma}\label{observation}
If there exist integers  $k$ and $m$ such that $\frac{kn}{c} \le m \le \frac{kn}{b}, \ \gcd(m,n)=1$ and $a \le \frac{b}{k}$, then Proposition \ref{maintheorem} holds.
\end{lemma}

\begin{proof}
Note that $m < \frac{kn}{b}$. Since  $ma < \frac{kn}{b} \time \frac{b}{k}=n$, the result follows from Lemma \ref{mainlemma} (1).
\end{proof}

\bigskip

In what follows, we assume that $s=\lfloor\frac{b}{a}\rfloor$. Then we have $ s\le \frac{b}{a} < s+1$ and
\begin{equation*}
\frac{n}{2a} < \frac{(s+1)n}{2b} < \cdots  < \frac{(2s-1)n}{2b} < \frac{sn}{b} \le \frac{n}{a}.
\end{equation*}
Since $b < \frac{n}{2}$, we have $\frac{n}{2b} =\frac{(2s-t)n}{2b}- \frac{(2s-t-1)n}{2b} > 1$, and then $[\frac{(2s-t-1)n}{2b}, \frac{(2s-t)n}{2b}]$ contains at least one integer for every $t \in [0, s-1]$.

Now we are ready to  give two sufficient conditions for  Proposition \ref{maintheorem} to hold. The first is  ``$s\ge 8$" (which follows from Lemmas \ref{about-s} and \ref{small-s})  and the other is ``$a=2$" (Lemma \ref{small-a}).

\bigskip

\begin{lemma}\label{about-s}
Suppose $s \ge 2$ and  $[\frac{(2s-2t-1)n}{2b}, \frac{(s-t)n}{b}]$ contains an integer  co-prime to  $n$ for some  $t \in [0, \lfloor \frac{s}{2} \rfloor-1 ]$. Then Proposition \ref{maintheorem} holds.
\end{lemma}

\begin{proof}
Suppose $M \in [\frac{(2s-2t-1)n}{2b}, \frac{(s-t)n}{b}]$ for some $t \in [0, \lfloor \frac{s}{2} \rfloor-1 ]$ and  $\gcd(M, n)=1$. Then $M < \frac{sn}{b} \le \frac{n}{a}\le\frac{n}{2}$. Note that $ \frac{n}{2a} < \frac{(2s-2t-1)n}{2b} < M < \frac{(s-t)n}{b} \le \frac{n}{a} $. Hence $|Ma|_n > \frac{n}{2}$. Also, since $\frac{(2s-2t-1)n}{2b} < M < \frac{(s-t)n}{b}$, we have $\frac{(2s-2t-1)n}{2} < Mb < (s-t)n$, so $|Mb|_n > \frac{n}{2}$. It follows from Lemma \ref{mainlemma} (2) that Proposition \ref{maintheorem} holds.
\end{proof}

\medskip
\begin{lemma}\label{small-s}
Suppose $s \ge 2$ and $[\frac{(2s-2t-1)n}{2b}, \frac{(s-t)n}{b}]$ contains no integers  co-prime to  $n$ for every  $t \in [0, \lfloor \frac{s}{2} \rfloor-1 ]$. Then the following results hold.
\begin{enumerate}

\item[(i).] $\frac{n}{2b} < 3$ (where $\frac{n}{2b}$ is the length of the interval  $[\frac{(2s-t_1-1)n}{2b}, \frac{(2s-t_1)n}{2b}]$ for each $t_1 \in [0, s-1]$).

\smallskip
\item[(ii).]  If $s \ge 4$, then
$[\frac{(2s-2t-1)n}{2b}, \frac{(s-t)n}{b}]$  contains exactly one integer for every  $t \in [0, \lfloor \frac{s}{2} \rfloor-1 ]$. Furthermore, $\frac{n}{2b} < 2$.

\smallskip
\item[(iii).]  Suppose that  $s \ge 4$, $x \in [\frac{(2s-2t-1)n}{2b}, \frac{(s-t)n}{b}]$  and $y \in [\frac{(2s-2t-3)n}{2b}, \frac{(s-t-1)n}{b}]$ for some $t \in [0 , \lfloor \frac{s}{2} \rfloor-2]$. Then  $\gcd(x , y,  n) = 1 $.

\smallskip
\item[(iv).]  Suppose that $s \ge 6$, $x \in [\frac{(2s-2t-1)n}{2b}, \frac{(s-t)n}{b}]$ and  $z \in [\frac{(2s-2t-5)n}{2b}, \frac{(s-t-2)n}{b}]$ for some $t \in [0 , \lfloor \frac{s}{2} \rfloor-3]$. Then $\gcd (x,z,n)>1 $ and $5 \mid \gcd (x, z, n)$. Furthermore, $z = x-5$ and $\frac{n}{2b} < \frac{7}{5}$.

\smallskip
\item[(v).]  $s \le 7.$
\end{enumerate}
\end{lemma}

\begin{proof}
(i). Since $[\frac{(2s-2t-1)n}{2b}, \frac{(s-t)n}{b}]$ contains no integers  co-prime to  $n$ for every  $t \in [0, \lfloor \frac{s}{2} \rfloor-1 ]$ and $n=p_1^{\alpha} \cdot p_2^{\beta}$, we have that $[\frac{(2s-2t-1)n}{2b}, \frac{(s-t)n}{b}]$ contains at most two integers and hence $\frac{n}{2b} = \frac{(s-t)n}{b} - \frac{(2s-2t-1)n}{2b}   < 3.$

Next, we assume that $s\ge 4$, so $\lfloor \frac{s}{2} \rfloor-1 \ge 1.$

(ii). Assume to the contrary that $[\frac{(2s-2t-1)n}{2b}, \frac{(s-t)n}{b}]$ contains two integers, say $x, x+1$, for some $t \in [0, \lfloor \frac{s}{2} \rfloor-1 ]$.  Then $\gcd(x, n)>1$ and $ \gcd(x+1, n)>1 $. Since $n=p_1^{\alpha} \cdot p_2^{\beta}$ and $\gcd(n, 6)=1$,  for every $z \in [x-3, x-1] \cup [x+2, x+4]$ we have $\gcd (z, n)=1$. Hence for every $t \in [0, \lfloor \frac{s}{2} \rfloor-1 ]$
$$\big ([x-3, x-1] \cup [x+2, x+4] \big) \cap [\frac{(2s-2t-1)n}{2b}, \frac{(s-t)n}{b}] = \emptyset.$$
If $t=0$, then $[x-3, x-1] \subset [\frac{(s-1)n}{b}, \frac{(2s-1)n}{2b}]$. Then $\frac{n}{2b} = \frac{(2s-1)n}{2b} - \frac{(s-1)n}{b} >2$. Hence $[\frac{(2s-3)n}{2b}, \frac{(s-1)n}{b}]$ contains at least two integers. It follows from (i) that $[\frac{(s-1)n}{b}, \frac{(2s-1)n}{2b}]$ contains at most three integers, so we have that $x-4, x-5 \in [\frac{(2s-3)n}{2b}, \frac{(s-1)n}{b}]$. Then  $\gcd(x-4, n)>1$ and $ \gcd(x-5, n)>1$, which together with $\gcd(x, n)>1$ and $ \gcd(x+1, n)>1 $ yield a contradiction to the assumption  that $n=p_1^{\alpha} \cdot p_2^{\beta}$ and $\gcd(n, 6)=1$. If $t \ge 1$, similarly, we can show that $x+5, x+6 \in [\frac{(2s-2t+1)n}{2b}, \frac{(s-t+1)n}{b}]$ and thus $\gcd(x+5, n)>1$ and $ \gcd(x+6, n)>1$, which yield a contradiction again.
Hence, $[\frac{(2s-2t-1)n}{2b}, \frac{(s-t)n}{b}]$  must contain exactly one integer for every  $t \in [0, \lfloor \frac{s}{2} \rfloor-1 ]$, and therefore, $\frac{n}{2b} = \frac{(s-t)n}{b} - \frac{(2s-2t-1)n}{2b} < 2$.

(iii). Since the length of $[\frac{(s-t)n}{b}, \frac{(2s-2t+1)n}{2b}]$ is  less than $2$ by (ii), this interval contains at most two integers for each $t \in [1, \lfloor \frac{s-1}{2} \rfloor]$. Since  $x \in [\frac{(2s-2t-1)n}{2b}, \frac{(s-t)n}{b}]$  and $y \in [\frac{(2s-2t-3)n}{2b}, \frac{(s-t-1)n}{b}]$,  we have $\gcd (x, n) >1$, and $\gcd(y, n)>1$. Note that $2 \le x-y \le 3$ and $\gcd(n , 6)=1$. We infer that $\gcd(n , x-y)=1$ and thus $\gcd(x,y,n)=1$. This proves (iii).

(iv). Assume that $s \ge 6$, and then $\lfloor \frac{s}{2} \rfloor-1 \ge 2.$ Assume that  $x \in [\frac{(2s-2t-1)n}{2b}, \frac{(s-t)n}{b}], \ y \in [\frac{(2s-2t-3)n}{2b}, \frac{(s-t-1)n}{b}]$ and $z \in [\frac{(2s-2t-5)n}{2b}, \frac{(s-t-2)n}{b}]$. Then  $\gcd (x, n) >1, \ \gcd(y, n) > 1$ and $\gcd(z, n) > 1$. By (iii) we have $\gcd(x,  y, n)=\gcd(y, z, n)=1$. Since $n=p_1^{\alpha} \cdot p_2^{\beta}$, we have that $\gcd(x, z, n)  = p_1^{\alpha_1}$ or $p_2^{\beta_1}$, where $1 \le \alpha_1 \le \alpha$ and $1 \le \beta_1 \le \beta$. Since $\gcd ( n, 6)=1$ and $3 \le x-z \le 6$, we have that $\gcd(x-z,  n) = x-z = 5$ and thus $5 \mid \gcd(x, z, n)$. Note that $[\frac{(2s-5)n}{2b}  , \frac{sn}{b}]$ contains exactly $6$ integers, so we infer that $ \frac{sn}{b} - \frac{(2s-5)n}{2b} < 7$. Hence $\frac{n}{2b} < \frac{7}{5}$, proving (iv).

(v). Assume to the contrary that $s \ge 8$. Then $\lfloor \frac{s}{2} \rfloor-1 \ge 3.$ Assume that  $x \in [\frac{(2s-2t-1)n}{2b}, \frac{(s-t)n}{b}], $ $\, \ y \in [\frac{(2s-2t-3)n}{2b}, \frac{(s-t-1)n}{b}], \ z \in [\frac{(2s-2t-5)n}{2b}, \frac{(s-t-2)n}{b}]$ and $w \in [\frac{(2s-2t-7)n}{2b}, \frac{(s-t-3)n}{b}]$. By (iv) we have $5 \mid x$ and $5 \mid y$, which is impossible since $2 \le x -y \le 3$.

This completes the proof.
\end{proof}

\bigskip
\begin{lemma}\label{small-a}
Proposition \ref{maintheorem} holds if $a=2$.
\end{lemma}

\begin{proof}
Note that $S=g((b+1)g)((n-b)g)((n-2)g).$ By Lemmas \ref{about-s} and \ref{small-s},  we may assume that $s \le 7$, and thus $b < (s+1)a \le 16$. Let $n=rb+b_0$, where $0\le b_0 \le  b-1$. Then $ n < (r+1)b$ and thus $r > \frac{n}{b} - 1 \ge \frac{1000}{16}-1 > 60$.

If $b=2t$, then let $m=(n-1)/2$, and clearly $\gcd(m,n)=1 $. Since  $|m|_n+|m(b+1)|_n+|m(n-b)|_n+|m(n-2)|_n=m+(m-t)+t+1=n$, we have $\ind(S)=1$.

If  $b=2t+1$, since $b\ge a=2$, we have  $1\le t<\frac{b}{2}<\frac{n}{4}$. Take $m_1=\frac{n-1}{2}$. Then $\{|m_1|_n, |m_1(n-2)|_n, |m_1(n-b)|_n, |m_1(b+1)|_n\}=\{1, \frac{n-1}{2}, n- \frac{n-b}{2}, n-t-1\}$, and $1 < t+1 < \frac{n-b}{2} < \frac{n-1}{2}$. Let $S'=g(c'g)((n-b')g)((n-a')g)$, where $a'=t+1, \ b'= \frac{n-b}{2}$ and $c'= \frac{n-1}{2}$. Since $\gcd(m_1, n)=1$ we have  $\ind(S)=\ind(S')$, and we shall show that $\ind(S')=1$.

Take $k'=\lceil\frac{n-b}{2b}\rceil$. It is easy to verify that $\frac{(k'+i)n}{c'} \le 2(k'+i)+1 \le \frac{(k'+i)n}{b'}$ for each $i \in [0,  2]$. Since $\gcd(n, 6)=1$ and $n=p_1^{\alpha} \cdot p_2^{\beta},$ we have  $\gcd(2(k'+i)+1,n)=1$ for some $i\in [0, 2]$. Now let $m$ be one of the integers in $\{2k'+1, 2k'+3, 2k'+5\}$ which is co-prime to $n$ and let $k=\frac{m-1}{2}$. Then $k \le b'$. We shall show that $ma' < n$, and then the result follows from Lemma \ref{mainlemma} (1).

If $r \equiv 1 \pmod 2$, then  $k'=\lceil\frac{n-b}{2b}\rceil=\frac{r+1}{2}$. Since $r > 60$, we have that $ma' \le (2k'+5)(t+1) = (r+6)(t+1) < r(2t+1)+b_0$ and we are done.

If $r \equiv 0 \pmod 2$, then  $k'=\lceil\frac{n-b}{2b}\rceil=\frac{r}{2}$. Since $r > 60$, we have that $ma' \le (2k'+5)(t+1) = (r+5)(t+1) < r(2t+1)+b_0$, and we are done.
\end{proof}

\bigskip

\section{Proof of Theorem \ref{maintheorem1}}\label{proof}

\bigskip

As mentioned in the last section, we need only prove Proposition \ref{maintheorem}. In view of Lemmas \ref{about-s}, \ref{small-s} and \ref{small-a}, from now on we may always assume that $s\le 7$ and $a\ge3$.

Let $k_1$ be the largest positive integer such that $\lceil \frac{(k_1-1)n}{c} \rceil = \lceil \frac{(k_1-1)n}{b} \rceil$ and $\frac{k_1n}{c} \le m < \frac{k_1n}{b}$. Since $\frac{bn}{c} \le n-1 < n = \frac{bn}{b}$ and $\frac{tn}{b} - \frac{tn}{c} = \frac{t(c-b)n}{bc} > 2$ for all $ t \ge b$, such integer $k_1$ always exists and $ k_1 \le b$.  Since $\lceil \frac{(k_1-1)n}{c} \rceil = \lceil \frac{(k_1-1)n}{b} \rceil$, we have
\begin{equation}\label{k1 is minimalty}
1 > \frac{(k_1-1)n}{b}- \frac{(k_1-1)n}{c} = \frac{(k_1-1)n(c-b)}{bc} =\frac{(k_1-1)n(a-1)}{bc}.
\end{equation}

\medskip

We now show that Proposition \ref{maintheorem} holds  through the following 3 propositions. The first one handles the case when $\lceil \frac{n}{c} \rceil < \lceil \frac{n}{b} \rceil$(i.e. $k_1=1$), and the others handle the case when $\lceil \frac{n}{c} \rceil = \lceil \frac{n}{b} \rceil$(i.e. $k_1\ge 2$).

\medskip
\begin{proposition}\label{case k = 1}
Suppose $\lceil \frac{n}{c} \rceil < \lceil \frac{n}{b} \rceil$ (i.e. there exists a positive integer $m_1$ such that $\frac{n}{c} \le m_1 <\frac{n}{b}$),  then Proposition \ref{maintheorem} holds.
\end{proposition}

\medskip
\begin{proposition}\label{case k > 1 and a is big}
Suppose $\lceil \frac{n}{c} \rceil = \lceil \frac{n}{b} \rceil$. Let $k_1$ be the largest positive integer such that $\lceil \frac{(k_1-1)n}{c} \rceil = \lceil \frac{(k_1-1)n}{b} \rceil$ and $\frac{k_1n}{c} \le m_1 <\frac{k_1n}{b}$ holds for some integer $m_1$.  If $k_1 >  \frac{b}{a}$, then Proposition \ref{maintheorem} holds.
\end{proposition}

\medskip
\begin{proposition}\label{case k > 1 and a is small}
Suppose $\lceil \frac{n}{c} \rceil = \lceil \frac{n}{b} \rceil$. Let $k_1$ be the largest positive integer such that $\lceil \frac{(k_1-1)n}{c} \rceil = \lceil \frac{(k_1-1)n}{b} \rceil$ and $\frac{k_1n}{c} \le m_1 <\frac{k_1n}{b}$ holds for some integer $m_1$. If $ k_1 \le  \frac{b}{a}$, then Proposition \ref{maintheorem} holds.
\end{proposition}

\bigskip

\begin{center}
 3.1. Proof of Proposition \ref{case k = 1}
\end{center}

\medskip

In this subsection, we assume that $\lceil \frac{n}{c} \rceil < \lceil \frac{n}{b} \rceil$. Let $m_1=\lceil \frac{n}{c} \rceil$. Then we have   $m_1-1 < \frac{n}{c} \le m_1 < \frac{n}{b}$. By Lemma \ref{mainlemma} (1), it suffices to find $k$ and $m$ such that $\frac{kn}{c} \le m <\frac{kn}{b}, \ \gcd(m,n)=1, \ 1 \le k \le b,$ and $ma < n$. So in what follows, we may always assume that $\gcd(n, m_1)>1$.

\medskip
\begin{lemma}\label{case k=1 lemma 2}
If $[\frac{n}{c}, \frac{n}{b}]$ contains at least two integers, then $\ind(S)=1$.
\end{lemma}

\begin{proof}
Since $a \le b$, by Lemma \ref{observation} we may assume every integer in $[\frac{n}{c}, \frac{n}{b}]$ is not co-prime to $n$. Since  $n=p_1^{\alpha}\cdot p_2^{\beta} $, we may assume that  $[\frac{n}{c}, \frac{n}{b}]$ contains exactly  two integers. Then $$m_1-1 < \frac{n}{c} \le m_1 < m_1+1 \le \frac{n}{b} < m_1+2,$$ and we also  have that $\ \gcd(n, m_1)>1$ and $\gcd(n, m_1+1)>1$. Since $n=p_1^{\alpha}\cdot p_2^{\beta}$ and $\gcd(n,6)=1$,  we infer that $m_1 \ge 10$ and $\gcd(2m_1+1,n)=1$. Then $n \ge (m_1+1)b \ge 11b$.

Note that
\begin{equation}\label{inequ 1}
2m_1-2 < \frac{2n}{c} \le 2m_1 < 2m_1+1 < 2m_1+2 \le \frac{2n}{b} < 2m_1+4.
\end{equation}
Let $m=2m_1+1$ and $ k=2$. We shall show that $ma < n$.

Since $1+c=a+b$, by \eqref{inequ 1} we have  $(2m_1-2)(b+a-1) = (2m_1-2)c < 2n < (2m_1+4)b$, and thus $(2m_1-2)(a-1)<6b.$ Since $a \ge 3$ and $m_1 \ge 10$, we have
\begin{align*}
ma = (2m_1+1)a =\frac{2m_1+1}{2m_1-2} \time \frac{a}{a-1} \time (2m_1-2)(a-1) < \frac{2\time 10+1}{2 \time 10 -2} \time \frac{3}{3-1} \time 6b < 11b \le n,
\end{align*}
and we are done.
\end{proof}

\medskip

By Lemma \ref{case k=1 lemma 2}, we may assume that $[\frac{n}{c}, \frac{n}{b}]$ contains exactly one integer $m_1$, and thus
$$m_1-1 < \frac{n}{c} \le m_1  < \frac{n}{b} < m_1+1.$$
Since $\gcd(n,6)=1$ and $(n,m_1)>1$,  we have either $m_1=5$ or $m_1=7$ or $m_1 \ge 10$.

Let $\ell$ be the smallest integer such that  $[\frac{ \ell n}{c}, \frac{\ell n}{b})$ contains at least three integers. Clearly, $\ell \ge 2$. Since $ \frac{n}{b} -m_1 < 1$ and $m_1-\frac{n}{c} < 1$, by using the minimality of $\ell$  we obtain that $\ell m_1-3 < \frac{\ell n}{c} < \frac{\ell n}{b} < \ell m_1+3$. Then $\frac{\ell n(c-b)}{bc} = \frac{\ell n}{b} - \frac{\ell n}{c} < (\ell m_1+3) - (\ell m_1-3)< 6$ and thus $\ell < \frac{6bc}{(c-b)n} = \frac{6b}{c-b}\time \frac{c}{n} < \frac{6b}{(a-1)(m_1-1)} \le \frac{6b}{(3-1)(5-1)} < b $.

\medskip

We claim that $[\frac{\ell n}{c},\frac{\ell n}{b})$ contains at most four integers. Assume to the contrary that  $[\frac{\ell n}{c},\frac{\ell n}{b})$ contains at least five integers, so $\ell m_1-3 < \frac{\ell n}{c} \le \ell m_1-2 < \ell m_1+2 < \frac{\ell n}{b} \le \ell m_1+3$. Then $\frac{2}{\ell} \le m_1 - \frac{n}{c}, \  \frac{n}{b}-m_1 \le \frac{3}{\ell}$. Since $\ell \ge 2$, we have  $ (\ell -1) ( m_1 - \frac{n}{c}) \ge  (\ell -1)\time \frac{2}{\ell} \ge 1$ and $  (\ell -1)(\frac{n}{b}-m_1) >  (\ell -1)\time \frac{2}{\ell} \ge 1$. Hence, $\frac{(\ell-1) n}{c} \le (\ell -1)m_1 - 1 < (\ell -1)m_1 + 1 < \frac{(\ell-1) n}{b}$, so  $[\frac{ (\ell-1) n}{c}, \frac{(\ell-1) n}{b})$ contains at least three integers, a contradiction to the minimality of $ \ell$.

By the above claim  we have  either
\begin{equation}\label{at most 4 integer big}
\ell m_1-2 < \frac{\ell n}{c} < \frac{\ell n}{b} < \ell m_1+3
\end{equation} or
\begin{equation}\label{at most 4 integer small}
\ell m_1-3 < \frac{\ell n}{c} < \frac{\ell n}{b} \le \ell m_1+2.
\end{equation}

\medskip

We remark that since $n=p_1^{\alpha} \cdot p_2^{\beta} $ and $[\frac{\ell n}{c},\frac{\ell n}{b})$ contains at least $3$ integers, one of them (say $m$) must be co-prime to $n$. If $ma<n$, then we are done by Lemma \ref{mainlemma} (1)(with $k=l<b$). Otherwise, for each integer $\delta \in [1, \ell-1]$, $\ [\frac{\delta n}{c},\frac{\delta n}{b})$ contains at most $2$ integers. We shall try to find an integer $m$ in one of those intervals such that $ma<n$ and this method will be used frequently in sequel.

\medskip

Recall that $[\frac{n}{c}, \frac{n}{b}]$ contains exactly one integer $m_1$,  $\ \gcd (m_1, n)>1 $ and $\ell \ge 2$ . We first deal with two special cases in the following two lemmas. More specifically, we shall show that if $[\frac{n}{c}, \frac{n}{b}]$ contains  $5$  or $7$, then Proposition \ref{case k = 1} holds.

\medskip

\begin{lemma}\label{case k=1 lemma 3}
If $4 < \frac{n}{c} \le 5  < \frac{n}{b} < 6$ and $5 \mid n$, then  $\ind(S)=1$.
\end{lemma}

\begin{proof}
Since $4 < \frac{n}{c} \le 5  < \frac{n}{b} < 6$, $n > 5b $. Note that $m_1=\lceil \frac{n}{c} \rceil =5.$

If $\ell=2$, since $[\frac{ \ell n}{c}, \frac{\ell n}{b})$ contains at least three integers,  we must have $8 < \frac{2n}{c} < 9 < 10 < 11 < \frac{2n}{b} < 12.$ Thus $  \frac{n}{6}< b <c < \frac{n}{4}$. Let $m = 9$ and $k = 2$. Then $ 9a = 9\time (c-b+1) \le 9 \time (\frac{n-1}{4}-\frac{n+1}{6}+1) < n,$ and we are done.

Next assume that $\ell \ge 3$. Since $[\frac{\ell n}{c},\frac{\ell n}{b})$ contains at least  three integers and $5\ell-3 < \frac{\ell n}{c} < \frac{\ell n}{b} \le 5\ell +3$, we can divide the proof into three cases.

\medskip
\noindent {\bf Case 1.} $5\ell + 2 \le \frac{\ell n}{b} < 5\ell +3$. Then $\frac{2}{\ell} \le \frac{n}{b}-5 \le \frac{3}{\ell}$.

If $\frac{\ell +1 }{2} \le \gamma \le \ell - 1$, then $\gamma(\frac{n}{b}-5) > \frac{\ell}{2} \cdot \frac{2}{\ell} = 1$ and thus $ \frac{\gamma n}{c} \le 5\gamma < 5\gamma +1 < \frac{\gamma n}{b}$. By the minimality of $\ell$ we infer that
\begin{equation}\label{gamma 1}
5\gamma-1 < \frac{\gamma n}{c} \le 5\gamma < 5\gamma+1 < \frac{\gamma n}{b} \le 5\gamma+2.
\end{equation}
Let $\gamma = \ell -1$. By \eqref{gamma 1}, we have $(5(\ell-1)-1)(b+a-1) = (5\ell-6)c < (\ell-1)n \le (5(\ell-1)+2)b$ and thus $(5\ell-6)(a-1) < 3b$.

First assume that  $\ell \ge 16$.   Let $k=\ell$ and let $m$ be an integer in $[\frac{\ell n}{c},\frac{\ell n}{b})$ which is co-prime to $n$. Since $5\ell + 2 \le \frac{\ell n}{b} < 5\ell +3,$ we have  that $m \le 5\ell +2$. Then
\begin{align*}
ma \le (5\ell+2)a =\frac{5\ell+2}{5\ell-6} \time \frac{a}{a-1} \time  (5\ell-6)(a-1) < \frac{5\time 16+2}{5 \time 16-6 } \time \frac{3}{3-1} \time 3b < 5b \le n,
\end{align*}
and we are done.

Next assume that $6 \le \ell \le 15$.

If  $\gcd(5(\ell-1)+1, n)=1$, let $m=5(\ell-1)+1$ and $ k= \ell -1$. Then by \eqref{gamma 1} $\frac{kn}{c} < m < \frac{kn}{b}$ and
\begin{align*}
(5(\ell-1)+1)a =\frac{5(\ell-1)+1}{5\ell-6} \time \frac{a}{a-1} \time  (5\ell-6)(a-1) < \frac{5\time 6-4}{5 \time 6-6 } \time \frac{3}{3-1} \time 3b < 5b \le n,
\end{align*}
as desired. Thus we may assume that $ \gcd( 5(\ell-1)+1 , n) > 1$.

If $ 13 \le \ell \le 15 $, applying \eqref{gamma 1} with $\gamma = 8$, we have $39 < \frac{8n}{c} < 40 < 41 < \frac{8n}{b} \le 42.$ Thus $\frac{4n}{21} \le b < c < \frac{8n}{39}$. Since $ \gcd( 5(\ell-1)+1 , n) > 1$ and $\gcd (5, n) >1$ and $n =p_1^{\alpha} \cdot p_2^{\beta},$ we have that $\gcd (41 , n)=1$. Let $m = 41$ and $ k= 8$. Then $41a = 41(c-b+1) < 41 (\frac{8n}{39}-\frac{4n}{21}+1) < n$, and we are done.

If $8 \le \ell \le 12 $, applying \eqref{gamma 1} with $\gamma = 7$, we have $34 < \frac{7n}{c} < 35 < 36 < \frac{7n}{b} \le 37,$ and so $\frac{7n}{37} \le b < c < \frac{7n}{34}$. Note that $\gcd (36 , n)=1$. Let $m = 36$ and $k= 7$. Then $36a = 36 \time (c-b+1) < 36 \time  (\frac{7n}{34}-\frac{7n}{37}+1) < n$, and we are done.

If $6 \le \ell \le 7 $, applying \eqref{gamma 1} with $\gamma = 4$, we have $19 < \frac{4n}{c} < 20 < 21 < \frac{4n}{b} \le 22,$ and so $\frac{2n}{11} \le b < c < \frac{4n}{19}$. As above we have $\gcd (21 , n)=1$. Let $m = 21$ and $k= 4$. Then $21a = 21(c-b+1) < 21 (\frac{4n}{19}-\frac{2n}{11}+1) < n$, and we are done.

Finally, assume that $\ell \le 5$.

If $4\le \ell \le 5 $, applying \eqref{gamma 1} with $\gamma = 3$, we have $14 < \frac{3n}{c} < 15 < 16 < \frac{3n}{b} \le 17,$ then $\frac{3n}{17} < b < c < \frac{3n}{14}$. Note that  $\gcd (16 , n)=1$. Now let $m = 16$ and $ k= 3$. Then $16a = 16(c-b+1) < 16 (\frac{3n}{14}-\frac{3n}{17}+1) < n$, and we are done.

If $ \ell = 3 $, we have $ \frac{3n}{c} < 15 < 16 < 17 < \frac{3n}{b} \le 18.$ If $\frac{3n}{c} \ge 14$, then  $c \le \frac{3n}{14}$ and $b \ge \frac{n}{6}$. Note that  $\gcd (16 , n)=1$. Let $m = 16$ and $  k= 3$. Then $16a = 16 \time (c-b+1) \le 16 \time (\frac{3n}{14}-\frac{n}{6}+1) < n$, and we are done. Now assume that $ \frac{3n}{c} < 14$, by \eqref{at most 4 integer big} we have $13 < \frac{3n}{c} < 14$. Applying \eqref{gamma 1} with $\gamma = 2$, we have $9 < \frac{2n}{c} < 10 < 11 < \frac{2n}{b} \le 12,$  and then $\frac{n}{6} < b < c < \frac{2n}{9}$. Note that either $\gcd(11, n)=1$ or $\gcd(14, n)=1$. Now let $k=2$ and $m=11$ if $\gcd(m, 11)=1$, or let $k=3$ and $m=14$ if $\gcd(m , 14)=1$. Then $ma \le 14\time (c-b+1) < 14 \time (\frac{2n}{9}-\frac{n}{6}+1) < n$.

This completes the proof of Case 1.

\medskip
\noindent {\bf Case 2.} $\frac{\ell n}{b} < 5\ell +2$ and $5\ell-3 < \frac{\ell n}{c} \le 5\ell - 2$.  This case can be proved in a similar way to Case 1.

\medskip
\noindent {\bf Case 3.} $\frac{\ell n}{b} < 5\ell +2$ and $ \frac{\ell n}{c} > 5\ell - 2$. Thus $5\ell-2 < \frac{\ell n}{c} \le 5\ell - 1 < 5\ell < 5\ell +1  < \frac{\ell n}{b} < 5\ell +2$. This implies that  every integer in $[\frac{\ell n}{c}, \frac{\ell n}{b})$ is less than $5\ell+2$. By the minimality of $\ell$, we must have one of the following holds.
\begin{itemize}
\item[(i)] $5(\ell-1)-1 < \frac{(\ell - 1) n}{c} < 5(\ell - 1)   < \frac{(\ell - 1) n}{b} \le 5(\ell - 1)+1$.

\item[(ii)] $5(\ell-1)-1 < \frac{(\ell - 1) n}{c} < 5(\ell - 1)   < 5(\ell - 1)+1 < \frac{(\ell - 1) n}{b} < 5(\ell - 1)+2$.

\item[(iii)] $5(\ell-1)-2 < \frac{(\ell - 1) n}{c} \le 5(\ell - 1)-1   < 5(\ell - 1) <  \frac{(\ell - 1) n}{b} \le 5(\ell - 1)+1$.
\end{itemize}

We  divide the proof into three subcases according to the above three situations.
\medskip

\noindent {\bf Subcase 3.1.} (i) holds. Then $ (5(\ell-1)-1)(b+a-1) = (5(\ell-1)-1)c < (\ell - 1) n  \le  (5(\ell - 1)+1) b$, so $(5\ell-6)(a-1) < 2b$.

If $\ell \ge 4$,   let $k=\ell$ and  $m$ be an integer in $[\frac{\ell n}{c}, \frac{\ell n}{b})$ which is co-prime to $n$. Note that $m \le 5\ell +1$. Then
\begin{align*}
ma \le (5\ell+1)a =\frac{5\ell+1}{5\ell-6} \time \frac{a}{a-1} \time  (5\ell-6)(a-1) < \frac{5\time 4+1}{5 \time 4-6 } \time \frac{3}{3-1} \time 2b < 5b \le n,
\end{align*}
so we are done. Therefore, we may assume that $ \ell =3$, so $13 < \frac{3n}{c} < 14 < 15 < 16 < \frac{3n}{b} < 17.$ Since $9 < \frac{2n}{c} < 10 < \frac{2n}{b} < 11,$ we have $ \frac{2n}{11} < b< c < \frac{2n}{9}$. Let $m=16$ and $\ k=3$. Then $ma = 16 \time (c-b+1) < 16 \time (\frac{2n}{9}-\frac{2n}{11}+1) < n$, and we are done.

\medskip

\noindent {\bf Subcase 3.2.} (ii) holds. Then $ (5(\ell-1)-1)(b+a-1) = (5(\ell-1)-1)c < (\ell - 1) n  <  (5(\ell - 1)+2) b$, so $(5\ell-6)(a-1) < 3b$.

If $\ell \ge 14$, then  let $k=\ell$ and let $m$ be an integer  in $[\frac{\ell n}{c}, \frac{\ell n}{b})$ which is co-prime to $n$. Note that $m \le 5\ell +1$. Thus
\begin{align*}
ma \le (5\ell+1)a =\frac{5\ell+1}{5\ell-6} \time \frac{a}{a-1} \time  (5\ell-6)(a-1) < \frac{5\time 14+1}{5 \time 14-6 } \time \frac{3}{3-1} \time 3b < 5b \le n,
\end{align*}
\noindent and we are done.

Next assume that $5 \le \ell \le 13$.
If $\gcd(5(\ell-1)+1, n)=1$, let $m=5(\ell-1)+1$ and $k= \ell -1$. Then
\begin{align*}
ma=(5(\ell-1)+1)a =\frac{5(\ell-1)+1}{5\ell-6} \time \frac{a}{a-1} \time  (5\ell-6)(a-1) < \frac{5\time 5-4}{5 \time 5-6 } \time \frac{3}{3-1} \time 3b < 5b \le n,
\end{align*}
and we are done. Hence we may assume that $ \gcd( 5(\ell-1)+1 , n) > 1$, which together with $\gcd (5, n) >1$ and $n =p_1^{\alpha} \cdot p_2^{\beta},$ implies  $\gcd (5\ell-1 , n)=1$. Now let $m = 5\ell-1$ and $ k= \ell$. Since $\frac{\ell n}{5\ell +2} < b < c < \frac{\ell n}{5\ell-2}$, we have $ ma = (5 \ell -1) (c-b+1) < (5 \ell -1)(\frac{\ell n}{5\ell-2} - \frac{\ell n}{5\ell +2} +1) < n$,  and we are done.

Finally, assume that $\ell \le 4$.

If $\ell=4$, then $14 < \frac{3n}{c} < 15 < 16 < \frac{3n}{b} < 17.$ It follows that $\frac{3n}{17} < b < c < \frac{3n}{14}$. Note that $\gcd(16, n)=1$. Let $m = 16$ and $k= 3$. Then $ma = 16 \time (c-b+1) \le 16 \time (\frac{3n-1}{14}-\frac{3n+1}{17}+1) < n$, and  we are done.

If $\ell=3$, we have $9 < \frac{2n}{c} < 10 < 11 < \frac{2n}{b} < 12$ and $13 < \frac{3n}{c} < 14 < 15 < \frac{3q}{b} < 18.$ Hence  $\frac{n}{6} < b < c < \frac{2n}{9}$. Now let $m=11$ and $k=2$ if $\gcd(11, n)=1$ or let $m=14$ and $k=3$ if $\gcd(n,14)=1$. Then  $ma \le 14a = 14 \time (c-b+1) \le 14 \time (\frac{2n-1}{9}-\frac{n+1}{6}+1) < n$, and we are done.

\medskip
\noindent {\bf Subcase 3.3.} (iii) holds. This subcase can be proved in a similar way to Subcase 3.2.

\end{proof}

\medskip

\begin{lemma}\label{case k=1 lemma 4}
If $6 < \frac{n}{c} \le 7  < \frac{n}{b} < 8$ and $7 \mid n$, then  $\ind(S)=1$.
\end{lemma}

\begin{proof}
Since  $6 < \frac{n}{c} \le 7  < \frac{n}{b} < 8$,  $n > 7b $. Note that $m_1=7$.

If $\ell =2$, then $12 < \frac{2n}{c} \le 13  < 14 < 15 < \frac{2n}{b} \le 16.$
Since $ 12(b+a-1) = 12c < 2n \le 16b$, we have that $12(a-1) < 4b$. If $ \gcd(15, n)=1$, let  $m=15$ and $k=2$; otherwise let $m=13$ and $k = 2$. Then
\begin{align*}
ma \le 15a =\frac{15}{12} \time \frac{a}{a-1} \time 12(a-1) < \frac{15}{12} \time \frac{3}{3-1} \time 4b <7b<n,
\end{align*}
and we are done.

Next assume that $\ell \ge 3$. Recall that $7\ell -3 < \frac{\ell n}{c} \le 7 \ell < \frac{\ell n}{b} < 7\ell +3$. We distinguish three cases according to where $\frac{\ell n}{c}$ locates.

\medskip
\noindent {\bf Case 1.} $7\ell-1 < \frac{\ell n}{c} \le 7\ell$.

Then $7\ell-1 < \frac{\ell n}{c} \le 7\ell < 7\ell+1 < 7\ell+2 < \frac{\ell n}{b} \le 7\ell+3$. It follows that $(7\ell-1)(b+a-1) = (7\ell-1)c < \ell n \le (7\ell+3)b$, hence $(7\ell-1)(a-1) < 4b$. If $\gcd(7\ell +2, n)>1$, then $\gcd(7\ell +1, n)=1$, so let $m=7\ell +1$ and $k= \ell$; otherwise  let $m=7\ell+2$ and $k= \ell$. Then
\begin{align*}
ma \le (7\ell+2)a =\frac{7\ell+2}{7\ell-1} \time \frac{a}{a-1} \time  (7\ell-1)(a-1) < \frac{7\time 3+2}{7 \time 3-1 } \time \frac{3}{3-1} \time 4b < 7b \le n,
\end{align*}
\noindent and we are done.

\medskip
\noindent {\bf Case 2.} $7\ell-2 < \frac{\ell n}{c} \le 7\ell-1$. Since  $[\frac{\ell n}{c}, \frac{\ell n}{b})$ contains at most four integers, we distinguish two subcases.

\medskip
\noindent {\bf Subcase 2.1.} $[\frac{\ell n}{c}, \frac{\ell n}{b})$ contains four integers.

Then $7\ell-2 < \frac{\ell n}{c} \le 7\ell-1 < 7\ell < 7\ell +2 \le \frac{\ell n}{b} \le 7\ell+3$. It follows that $\frac{1}{\ell} \le 7-\frac{n}{c} < \frac{2}{\ell} < \frac{n}{b}-7 \le \frac{3}{\ell}$. Therefore, $(\ell-1)(\frac{n}{b}-7) > (\ell -1)\frac{2}{\ell} > 1 $. By the minimality of $\ell$ we have $$7(\ell-1)-1 < \frac{(\ell-1)n}{c} < 7(\ell-1) < 7(\ell-1)+1 < \frac{(\ell-1)n}{b} \le 7(\ell-1)+2.$$ Then $(7\ell-8)(b+a-1) = (7\ell-8)c < (\ell-1)n \le (7\ell-5)b $; hence $(7\ell-8)(a-1) < 3b$.

First assume that  $\ell \ge 4$.  If $\gcd(7\ell +1, n)>1$, then $\gcd(7\ell -1, n)=1$, so let $m=7\ell -1$ and $k=\ell$; otherwise let   $m=7\ell+1$ and $k = \ell$.
\begin{align*}
ma \le (7\ell+1)a =\frac{7\ell+1}{7\ell-8} \time \frac{a}{a-1} \time  (7\ell-8)(a-1) < \frac{7\time 4+2}{7 \time 4-8 } \time \frac{3}{3-1} \time 3b < 7b \le n,
\end{align*}
and we are done.

Next assume that $\ell = 3$. Then $13 < \frac{2n}{c} < 14 < 15 < \frac{2n}{b} \le 16$.
Note that $\frac{4n}{c} < 27 < 28 < \frac{4n}{b}$,  $\gcd(27, n)=1$ and $\frac{n}{8} < b <  c < \frac{2n}{13}$. Let $m = 27$ and $k = 4$. Then $ ma = 27 \time (c-b+1) \le 27 \time (\frac{2n-1}{13} - \frac{n+1}{8} +1) <n$, and we are done.

\medskip
\noindent {\bf Subcase 2.2.} $[\frac{\ell n}{c}, \frac{\ell n}{b})$ contains three integers.

Then $ 7\ell-2 < \frac{\ell n}{c} \le 7\ell-1 < 7\ell < 7\ell +1 < \frac{\ell n}{b} \le 7\ell+2.$ It follows that $(7\ell-2)(b+a-1) = (7\ell-2)c < \ell n \le (7\ell+2)b $; hence $(7\ell-2)(a-1) < 4b$.

First assume that  $\ell \ge 4$.  If $\gcd(7\ell +1, n)>1$, then $\gcd(7\ell -1, n)=1$, so let $m=7\ell -1$ and $k=\ell$; otherwise let   $m=7\ell+1$ and $k = \ell$. Then
\begin{align*}
ma \le (7\ell+1)a =\frac{7\ell+1}{7\ell-2} \time \frac{a}{a-1} \time  (7\ell-2)(a-1) < \frac{7\time 4+2}{7 \time 4-2 } \time \frac{3}{3-1} \time 4b < 7b \le n,
\end{align*}
and we are done.

Next assume that $\ell = 3$. Then $19 < \frac{3n}{c} < 20 < 21 < 22 < \frac{3n}{b} \le 23$. If $\gcd(n, 20)=1$, let $k=3$ and $m=20$; otherwise let $m = 22$ and $k = 3$. Note that  $\frac{3n}{23} < b < c < \frac{3n}{19}$. Then $ ma \le 22 \time (c-b+1) \le 22 \time (\frac{3n-1}{19} - \frac{3n+1}{23} +1) < n$, and we are done.

\medskip
\noindent {\bf Case 3.}  $7\ell-3 < \frac{\ell n}{c} \le 7\ell-2$. As in Case 2, we  distinguish two subcases.

\medskip
\noindent {\bf Subcase 3.1.}  $[\frac{\ell n}{c}, \frac{\ell n}{b})$ contains four integers.

It follows that   $7\ell-3 < \frac{\ell n}{c} \le 7\ell-2 < 7\ell < 7\ell +1 \le \frac{\ell n}{b} \le 7\ell+2$. The proof is similar to Subcase 2.1.

\medskip
\noindent {\bf Subcase 3.2.}  $[\frac{\ell n}{c}, \frac{\ell n}{b})$ contains three integers.

It follows that $7\ell < \frac{\ell n}{b} < 7\ell+1$.  Similar to Case 1, We can prove the result.

This completes the proof.
\end{proof}

Now we are in a position to prove Proposition \ref{case k = 1}.

\medskip

\noindent {\bf Proof of Proposition \ref{case k = 1}.}

Recall that either $m_1=5$ or $m_1=7$ or $m_1 \ge 10$. By Lemmas \ref{case k=1 lemma 3} and \ref{case k=1 lemma 4} we may assume  $m_1 \ge 10$. Then $n \ge m_1b \ge 10b$. Let $k=\ell$ and let $m$ be one of the integers  in $[\frac{\ell n}{c},\frac{\ell n}{b})$ which is co-prime to $n$. Recall that we have either \eqref{at most 4 integer big} holds or \eqref{at most 4 integer small} holds.

If \eqref{at most 4 integer big} holds, then $(\ell m_1-2)(b+a-1) = (\ell m_1-2)c < \ell n \le (\ell m_1+3)b$, so $(\ell m_1-2)(a-1) < 5b.$ Note that $m \le \ell m_1+2$ and $\ell \ge 2$, then
\begin{align*}
ma  \le ( \ell m_1+2 )a =\frac{\ell m_1+2}{\ell m_1-2} \time \frac{a}{a-1} \time  (\ell m_1-2)(a-1) < \frac{2\time 10+2}{2 \time 10-2 } \time \frac{3}{3-1} \time 5b < 10b \le n,
\end{align*}
and we are done.

If \eqref{at most 4 integer small} holds, then $(\ell m_1-3)(b+a-1) = (\ell m_1-3)c < \ell n \le (\ell m_1+2)b$, hence $(\ell m_1-3)(a-1) < 5b.$ Note that $m \le \ell m_1+1$ and $\ell \ge 2$. Then
\begin{align*}
ma  \le ( \ell m_1+1 )a =\frac{\ell m_1+1}{\ell m_1-3} \time \frac{a}{a-1} \time  (\ell m_1-3)(a-1) < \frac{2\time 10+1}{2 \time 10-3 } \time \frac{3}{3-1} \time 5b < 10b \le n,
\end{align*}
and we are done. \qed

\bigskip
\begin{center}
 3.2. Proof of Proposition \ref{case k > 1 and a is big}
\end{center}
\medskip

In this subsection, we always assume that $\lceil \frac{n}{c} \rceil = \lceil \frac{n}{b} \rceil$, so  $k_1 \ge 2$, and we also assume that $k_1 > \frac{b}{a}$.
\medskip

\noindent {\bf Proof of Proposition \ref{case k > 1 and a is big}}

We distinguish the proof to two cases.\\
\noindent {\bf Case 1.} $k_1=2$.

If $a-1 \ge \frac{b}{2}$, then $$\frac{(c-b)(k_1-1)n}{cb} = \frac{(a-1)n}{bc}  =\frac{ a-1}{b} \time \frac{ n}{c}>\frac{1}{2} \time 2=1,$$ which is a contradiction to \eqref{k1 is minimalty}.
Hence,  $\frac{b}{2}-1 <a-1 < \frac{b}{2}.$

If $b=2\ell$, then $\ell-1 < a-1 < \ell$, which is impossible. Therefore, we must have that $b=2\ell +1$, and thus $a=\ell+1$. Now $c = a+b-1 = 3\ell+1$ and $c-b = a-1 = \ell$.

If $n\ge 3c$, then $$\frac{(c-b)(k_1-1)n}{cb} = \frac{\ell n}{(2\ell+1)c} \ge \frac{3\ell}{2\ell+1} \ge  1,$$ which is
a contradiction to \eqref{k1 is minimalty}.

Thus $n < 3c$, so  we assume that $n=2c+\ell_0$ for some $\ell_0$ odd since $\gcd (n, 6)=1$. If $\ell_0=1$, then $n = 2c+1 = 2(3\ell+1)+1 = 3(2\ell+1)$, a contradiction (since $n$ is not divisible by 3). If $\ell_0=3$, then $n=2c+3$ and thus $\frac{n}{c}\le 3 < \frac{n}{b} = \frac{6\ell+5}{2\ell+1}.$ This implies that $\lceil \frac{n}{c} \rceil < \lceil \frac{n}{b} \rceil$, a contradiction. If $\ell_0\ge 5$, then $n\geq 2c+5=6\ell+7.$ Now $$\frac{(c-b)(k_1-1)n}{cb} = \frac{\ell n}{(2\ell+1)(3\ell+1)} \ge  \frac{\ell(6\ell+7)}{6\ell^2+5\ell+1}>1,$$ a contradiction to \eqref{k1 is minimalty}.

\medskip
\noindent {\bf Case 2.} $k_1\ge 3$.

If $a-1\ge \frac{b}{k_1}$, then $\frac{(c-b)(k_1-1)n}{cb} \ge  \frac{b(k_1-1)n}{cbk_1} > \frac{2(k_1-1)}{k_1} > 1,$ a contradiction.

Thus, we have that $\frac{b}{k_1}+1 > a > \frac{b}{k_1}$. Assume that $b=k_1 \ell + k_0$ for some $0\le k_0 < k_1$ and $\ell \ge 1$. Note that if $k_0=0$, then $\ell+1 > a > \frac{b}{k_1} = \ell$, a contradiction. Therefore, we have that $1\le k_0 < k_1$. Then $a=\ell+1$, so $c= a+b-1 = (k_1+1)\ell+k_0$ and also $c-b =a-1 =\ell$. Now \eqref{k1 is minimalty} reduces to the following.
\begin{equation}\label{minimality of k1 case2}
\frac{(c-b)(k_1-1)n}{cb}=\frac{(k_1-1)\ell n}{(k_1\ell+k_0)c}<1.
\end{equation}

If $\ell \ge 2$, then $k_1(\ell-1)-2\ell+1 \ge 3(\ell-1)-2\ell+1 = \ell-2 \ge 0$. This together with $\frac{n}{c}>2$ and $k_1-1\ge k_0$ implies that $\frac{(c-b)(k_1-1)n}{cb} \ge \frac{2(k_1-1)\ell}{k_1\ell+k_1-1} = \frac{2\ell k_1-2\ell}{k_1(\ell+1)-1} = 1+\frac{k_1(\ell-1)-2\ell+1}{k_1(\ell+1)-1} \ge 1$, which is a contradiction to \eqref{minimality of k1 case2}.

Finally, assume that $\ell=1$, so $a=2$. Therefore, the proposition follows from Lemma  \ref{small-a}.

This completes the proof.  \qed

\bigskip

\begin{center}
 3.3. Proof of Proposition \ref{case k > 1 and a is small}
\end{center}

\medskip

In this subsection, we always assume that $\lceil \frac{n}{c} \rceil = \lceil \frac{n}{b} \rceil$, so $k_1 \ge 2$, and we also assume that $k_1 \le \frac{b}{a}$. Since $a \le \frac{b}{k_1}$, by Lemma \ref{observation} we may assume that $\gcd(m_1, n)>1$ for every  $ m_1 \in [\frac{k_1n}{c}, \frac{k_1n}{b})$.

%

\bigskip
Since $s=\lfloor\frac{b}{a}\rfloor$, we have $2 \le k_1 \le s\le \frac{b}{a}<s+1$. By Lemma \ref{about-s} we may assume that for every  $t \in [0, \lfloor \frac{s}{2} \rfloor-1 ]$
\begin{equation}\label{no integer co-prime to n}
[\frac{(2s-2t-1)n}{2b}, \frac{(s-t)n}{b}] \mbox{ contains no integers  co-prime to } n.
\end{equation}

We divide the proof of Proposition \ref{case k > 1 and a is small} into the following  few lemmas.

\medskip

Recall that by Remark \ref{remark} (1), we may always assume that $n > 1000$. The next lemma provides an upper bound for $n$, which will be used frequently to obtain a contradiction by showing that $n \le 1000$.

\medskip

\begin{lemma}\label{bound of n}
If $u < \frac{n}{c} < \frac{n}{b} < v$ for some real numbers $u , v $ and $ u(k_1-1) > s+1$, then $$n < \frac{uv(k_1-1)(s+1)}{u(k_1-1)-(s+1)}.$$
\end{lemma}

\begin{proof} Assume to the contrary that $n \ge  \frac{uv(k_1-1)(s+1)}{u(k_1-1)-(s+1)}$. Then $ b > \frac{n}{v} \ge \frac{u(k_1-1)(s+1)}{u(k_1-1)-(s+1)}.$ Since $  \frac{b}{a} < s+1$,  $a > \frac{b}{s+1}$. Therefore,
\begin{align*}
& \frac{(k_1-1)n(a-1)}{bc} = \frac{k_1-1}{b} \time \frac{n}{c} \time (a-1) > \frac{k_1-1}{b}\time u \time ( \frac{b}{s+1} -1) \\ = & \frac{u(k_1-1)}{s+1} \time \frac{b-(s+1)}{b} \ge  \frac{u(k_1-1)}{s+1} \time \frac{\frac{u(k_1-1)(s+1)}{u(k_1-1)-(s+1)}-(s+1)}{\frac{u(k_1-1)(s+1)}{u(k_1-1)-(s+1)}} \\ = & 1,
\end{align*}
yielding a contradiction to \eqref{k1 is minimalty}.
\end{proof}

\bigskip

\begin{lemma}\label{k < 4}
If the assumption is as in Proposition \ref{case k > 1 and a is small}, then $k_1 \le 3$.
\end{lemma}

\begin{proof}
 Assume to the contrary that $k_1 \ge 4$. Recall that $k_1 \le s\le 7$.  We distinguish three cases according to the value of $k_1$.

\medskip
\noindent {\bf Case 1.} $ 6 \le k_1 \le 7$. Then $6 \le k_1 \le s \le 7 $. By Lemma \ref{small-s} (iv)  we have that  $5 \mid n$ and $\frac{n}{b} < \frac{14}{5}.$  Thus  $2 < \frac{n}{c} < \frac{n}{b} < \frac{14}{5}.$ Applying Lemma \ref{bound of n} with $u=2$ and $v=\frac{14}{5}$, we infer that $n < 112$, yielding a contradiction Remark \ref{remark} (1).

\medskip
\noindent {\bf Case 2.} $ k_1=5$. Then $5 = k_1 \le s \le 7$.  By Lemma \ref{small-s} (ii) we have $\frac{n}{b} < 4$.

If $3 <  \frac{n}{b} < 4,$ since $\lceil \frac{n}{c} \rceil = \lceil \frac{n}{b} \rceil$, we have $3 < \frac{n}{c} < \frac{n}{b} < 4.$ Applying Lemma \ref{bound of n} with $u=3$ and $v=4$, we infer that $n < 84$, yielding a contradiction again.

Now assume that $2 <  \frac{n}{b} < 3.$ Since $\lceil \frac{n}{c} \rceil = \lceil \frac{n}{b} \rceil$, we have $2 < \frac{n}{c} < \frac{n}{b} < 3.$ If $s \le 6$. Then applying Lemma \ref{bound of n} with $u=2$ and $v=3$, we infer that $n < 168$, yielding a contradiction. If $s=7$, by Lemma \ref{small-s} (ii) and (iv) we conclude that $[\frac{(2s-2t-1)n}{2b}, \frac{(s-t)n}{b}]$ contains exactly one integer for each $t \in [0, \lfloor \frac{s}{2} \rfloor-1 ]$ and  the integer in $[\frac{9n}{2b},  \frac{5n}{b}]$ is divisible by $5$. Note that $ 9 < \frac{9n}{2b} < \frac{27}{2}$ and $10 < \frac{5n}{b} < 15$. We infer that $ 9 < \frac{9n}{2b} \le 10 < \frac{5n}{b} < 11$. By the definition of $k_1$, $[\frac{5n}{c}, \frac{5n}{b})$ contains an integer, so  we conclude that $  \frac{5n}{c} \le 10 < \frac{5n}{b} < 11$. Thus $\frac{n}{c} \le 2$, yielding a contradiction to that $c < \frac{n}{2}$.

\medskip
\noindent {\bf Case 3.} $ k_1=4$. Then $4 = k_1 \le s \le 7$.  By Lemma \ref{small-s} (ii) we have $\frac{n}{b} < 4$. We divide the proof into two subcases.

\medskip
\noindent {\bf Subcase 3.1.} $3 <  \frac{n}{b} < 4.$ Since $\lceil \frac{n}{c} \rceil = \lceil \frac{n}{b} \rceil$, we have $3 < \frac{n}{c} < \frac{n}{b} < 4.$ Applying Lemma \ref{bound of n} with $u=3$ and $v=4$, we infer that $n < 126$, a contradiction.

\medskip
\noindent {\bf Subcase 3.2.} $2 <  \frac{n}{b} < 3.$ Since $\lceil \frac{n}{c} \rceil = \lceil \frac{n}{b} \rceil$, we have $2 < \frac{n}{c} < \frac{n}{b} < 3,$  so $8 < \frac{4n}{c} \le m_1 < \frac{4n}{b} <12.$ Since $\gcd(n, 6)>1$ and $\gcd(n, m_1)>1$, we infer that $m_1\in [10, 11]$, so we have either $ 9 < \frac{4n}{c} \le 10 <  \frac{4n}{b} < 12$ or $ 10 < \frac{4n}{c} \le 11  < \frac{4n}{b} < 12$.

First assume that  $ 9 < \frac{4n}{c} \le 10 <  \frac{4n}{b} < 12$. If $s \le 5$, then applying Lemma \ref{bound of n} with $u= \frac{9}{4} $ and $v=3$, we infer that $n < 162$, a contradiction. Now assume that  $s=6$ or $s=7$.  We shall show that $ \frac{4n}{b} \le 11$. Assume to the contrary that  $\frac{4n}{b} >11$. Then  $\frac{11}{4} < \frac{n}{b} < 3$ and $[10,11]\subset  [\frac{4n}{c}, \frac{4n}{b})$. Since $\gcd(n, m_1)>1$ for every integer $m_1\in [\frac{4n}{c}, \frac{4n}{b})$, we have $\gcd(n, 10) > 1$ and $\gcd(n, 11) > 1$. Since $n=p_1^{\alpha} \cdot p_2^{\beta}$, we infer that $n = 5^{\alpha} \cdot 11^{\beta}$. Since $ 15 < \frac{121}{8} < \frac{11n}{2b} < \frac{6n}{b} < 18$, we conclude that  $[\frac{11n}{2b},   \frac{6n}{b}]$ contains an integer which is co-prime to $n$, a contradiction to \eqref{no integer co-prime to n}. Thus we must have $  \frac{4n}{b} \le 11$. Then  $\frac{5}{2} < \frac{n}{b} \le \frac{11}{4}$, which  implies that   $11 < \frac{45}{4} < \frac{9n}{2b} < \frac{5n}{b} < \frac{55}{4} < 14$. Suppose $x \in [\frac{9n}{2b}, \frac{5n}{b}]\subset [12, 13]$. By \eqref{no integer co-prime to n} we have $\gcd(n, x)>1$. Since $\gcd(n, 6) =1$, we infer that $x=13$ and hence $ 12 < \frac{9n}{2b} \le 13 \le \frac{5n}{b} < 14$. Thus $ \frac{8}{3} < \frac{n}{b} \le \frac{11}{4}$, which implies that $ \frac{11n}{2b} \le \frac{121}{8} < 16 < \frac{6n}{b}$. Therefore, $16 \in [\frac{11n}{2b} , \frac{6n}{b}]$. By \eqref{no integer co-prime to n}, we obtain that $\gcd(n, 16) > 1$, a contradiction to  $\gcd(n, 6)=1$.

Next assume that $ 10 < \frac{4n}{c} \le 11  < \frac{4n}{b} <12$. If $s \le 6$, then applying Lemma \ref{bound of n} with $u= \frac{5}{2} $ and $v=3$, we infer that $n < 315$. Now assume that  $s=7$.  Note that $\frac{11}{4} \le \frac{n}{b} < 3$. Then $ 12 < \frac{99}{8} \le \frac{9n}{2b} < \frac{5n}{b} < 15$ and thus $[\frac{9n}{2b}, \frac{5n}{b}] \subset [13, 14]$. However, by  Lemma \ref{small-s} (iv),  we obtain that   $[\frac{9n}{2b},  \frac{5n}{b}]$ contains an integer which is divisible by $5$,  yielding a contradiction.
\end{proof}

\medskip
\begin{lemma}\label{k = 3}
If $k_1=3$, then $\ind(S)=1$.
\end{lemma}

\begin{proof}

We remark that since $k_1=3$, we conclude that $[ \frac{2n}{c}, \frac{2n}{b})$ contains no integers and $[ \frac{k_2n}{c}, \frac{k_2n}{b})$ contains at least one integer where $k_2\ge 3$.

Note that $3= k_1 \le s \le 7$. By Lemma \ref{small-s} (i) we have   $ \frac{n}{b} < 6$.  We distinguish three cases according to the value of $\frac{n}{b}$.

\noindent {\bf Case 1.} $4 < \frac{n}{b} < 6.$ By Lemma \ref{small-s} (ii) we must have  $s=3$. Since $\lceil \frac{n}{c} \rceil = \lceil \frac{n}{b} \rceil$, we infer that $ 4 < \frac{n}{c} < \frac{n}{b} < 6$. Applying Lemma \ref{bound of n} with $u=4$ and $v=6$, we infer that $n < 48$, a contradiction to Remark \ref{remark} (1).

\medskip
\noindent {\bf Case 2.} $3 < \frac{n}{b} < 4.$ By Lemma \ref{small-s} (iv), we must have  $ s \le 5$.

Since $\lceil \frac{n}{c} \rceil = \lceil \frac{n}{b} \rceil$, we have $3 < \frac{n}{c} < \frac{n}{b} < 4$  and thus $9 < \frac{3n}{c} \le m_1 < \frac{3n}{b} <12.$ We infer that $m_1\in [10, 11]$, so we have either  $ 9 < \frac{3n}{c} \le 10 <  \frac{3n}{b} < 12$ or $ 10 < \frac{3n}{c} \le 11  < \frac{3n}{b} < 12$. If $s \le 4$, then applying Lemma \ref{bound of n} with $u= 3 $ and $v=4$, we infer that $n < 120$, a contradiction. Next assume that $s=5$.

Suppose that $ 9 < \frac{3n}{c} \le 10 <  \frac{3n}{b} < 12$. We will show that $\frac{3n}{b} < 11$. Assume to the contrary that $\frac{3n}{b} \ge 11$, then $\frac{2n}{c} \le \frac{20}{3} < 7 < \frac{22}{3} < \frac{2n}{b}$, which is impossible since $[ \frac{2n}{c}, \frac{2n}{b})$ contains no integers. Therefore, we have that $\frac{3n}{b} < 11$, and then $  \frac{10}{3} < \frac{n}{b} < \frac{11}{3}$. Hence $15  < \frac{9n}{2b} < \frac{5n}{b} < \frac{55}{3} < 19$ and $11 < \frac{35}{3} < \frac{7n}{2b} < \frac{4n}{b} < \frac{44}{3} < 15$. Thus $[\frac{9n}{2b}, \frac{5n}{b}] \subset [ 16, 18]$ and $[\frac{7n}{2b}, \frac{4n}{b}] \subset [ 12, 14]$. Suppose $x\in [\frac{9n}{2b}, \frac{5n}{b}]$ and $y \in [\frac{7n}{2b}, \frac{4n}{b}]$. Then $5 \nmid xy$. Since $n=p_1^{\alpha} \cdot p_2^{\beta}$, by Lemma \ref{small-s} (iii) we may assume that $p_1 \mid x$ and $p_2 \mid y$. Hence $p_1p_2 \mid xy$. Note that $m_1=10$. Since $\gcd(n, m_1)>1$ and $\gcd(n, 6)=1$, we have $5 \mid n$. Hence $5 \in \{p_1, p_2\}$ and thus $5 \mid xy$, yielding a contradiction.

Suppose that $ 10 < \frac{3n}{c} \le 11  < \frac{3n}{b} < 12$. Then applying Lemma \ref{bound of n} with $u= \frac{10}{3} $ and $v=4$, we infer that $n < 240$, a contradiction.

\medskip
\noindent {\bf Case 3.} $2 < \frac{n}{b} < 3.$ Since $\lceil \frac{n}{c} \rceil = \lceil \frac{n}{b} \rceil$, we have $2 < \frac{n}{c} < \frac{n}{b} <3,$ and thus $6 < \frac{3n}{c} \le m_1 <  \frac{3n}{b} < 9.$ Since $\gcd(n, 6)>1$ and $\gcd(n, m_1)>1$, we infer that $m_1=7$ and thus $6 < \frac{3n}{c} \le 7 < \frac{3n}{b} < 8$.  Therefore, $7 \mid n$ and $  \frac{7}{3} < \frac{n}{b} < \frac{8}{3}$.

First assume that $s \ge 5$.  Note that $ \frac{9n}{2b} < 12 $. We will show that $\frac{5n}{b} < 12$. If $ \frac{5n}{b} \ge 12$, then $12 \in [\frac{9n}{2b}, \frac{5n}{b} ]$. By \eqref{no integer co-prime to n}, we have $\gcd(n, 12) >1$, a contradiction to that $\gcd(n , 6) =1$. Therefore   $\frac{7}{3} < \frac{n}{b} < \frac{12}{5}$.

If $s=7$, by  Lemma \ref{small-s} (iv),  we have that the integer in $[\frac{13n}{2b},  \frac{7n}{b}]$ is divisible by $5$, which is impossible because   $ 15 <  \frac{13n}{2b} < \frac{7n}{b}  < 17$.

If $5 \le s \le 6$, we have that $ 8 < \frac{49}{6} < \frac{7n}{2b}  < \frac{4n}{b} < \frac{48}{5} < 10$. Hence $9 \in  [\frac{7n}{2b}  , \frac{4n}{b} ]$. By \eqref{no integer co-prime to n}, we have $\gcd(n, 9) >1$, a contradiction again.

Next assume that $s=4$. We have that $ 8 < \frac{49}{6} < \frac{7n}{2b}  < \frac{4n}{b} < \frac{32}{3} < 11$ and thus $[\frac{7n}{2b} , \frac{4n}{b}]\subset [9, 10]$. By \eqref{no integer co-prime to n} and $\gcd(n,6)=1$, we have $[\frac{7n}{2b} , \frac{4n}{b}]= \{ 10\}$ and thus $ 9  < \frac{7n}{2b} \le 10 \le  \frac{4n}{b} < 11$, so $\frac{18}{7} < \frac{n}{b} < \frac{8}{3}$. If $\frac{4n}{c} \le  10$, then $\frac{2n}{c} \le  5 < \frac{2n}{b} $, which is impossible since  $[ \frac{2n}{c}, \frac{2n}{b})$ contains no integers. So we must have that  $ 10 < \frac{4n}{c}  < \frac{4n}{b} < 11$, yielding a contradiction to the assumption  that $[ \frac{4n}{c}, \frac{4n}{b})$ contains an integer.

Finally, assume that $s=3$.  If $\frac{9}{4} < \frac{n}{c} < \frac{n}{b} < 3$, then applying Lemma \ref{bound of n} with $u= \frac{9}{4} $ and $v=3$, we infer that $n < 108$, a contradiction. Thus  $\frac{n}{c} \le \frac{9}{4},$ and then $\frac{2n}{c} \le \frac{9}{2} < 5$. Since $[ \frac{2n}{c}, \frac{2n}{b})$ contains no integers, we have $\frac{2n}{c} < \frac{2n}{b} \le 5$. Since $\frac{n}{b} > \frac{7}{3} > \frac{9}{4}$, we have that $\frac{4n}{c} \le 9 < \frac{4n}{b}$. Let $m=9$ and $k=4$. Since $\frac{2n}{5} < b < c < \frac{n}{2}$,  $ma = 9 \time ( c- b+1) \le 9\time (\frac{n-1}{2}-\frac{2n}{5} +1 ) < n$, so the lemma follows from Lemma \ref{mainlemma} (1).
\end{proof}

\begin{lemma}\label{k = 2 and 5 | n}
Let $k_1=2$,  $4 < \frac{2n}{c} \le 5 < \frac{2n}{b} < 6$ and $a \le \frac{b}{2}$. Then $\ind(S)=1$.
\end{lemma}

\begin{proof}
Note that $m_1=5$ and $b\ge 2a\ge 6$. Since $\gcd(n, m_1)>1$ we have $5 \mid n$. Since $4 < \frac{2n}{c} \le 5 < \frac{2n}{b} < 6$, we obtain that $\frac{n}{3} \le b < \frac{2n}{5} \le c < \frac{n}{2}$.

Recall that $n > 1000$. We shall show that either there exist positive integers $k, m$ satisfying the condition of Lemma \ref{mainlemma} (1), or  there exists a positive integer $M$  satisfying the condition of Lemma \ref{mainlemma} (2), and then the lemma follows from Lemma \ref{mainlemma}.

By the definition of $k_1$, we conclude that $[\frac{k_2n}{c}, \frac{k_2n}{b})$ contains at least one integer for each $k_2 \ge k_1=2$. Note that $6 < \frac{3n}{c} < \frac{3n}{b} < 9$. We distinguish   three cases.

\medskip
\noindent {\bf Case 1.} $7 < \frac{3n}{c} \le 8 < \frac{3n}{b} < 9.$ Then $\frac{n}{3} < b < \frac{3n}{8} \le c < \frac{3n}{7}$.

Note that $\gcd(n, 8)=1$. Let $m = 8$ and $k=3 \,  ( \le b=6)$. Since $\frac{n}{3} < b < c < \frac{3n}{7}$,  $ma = m (c-b+1) \le 8 \time (\frac{3n-1}{7} - \frac{n+1}{3} +1) < n$, and we are done.

\medskip
\noindent {\bf Case 2.} $6 < \frac{3n}{c} \le 7 < \frac{3n}{b} < 8.$ Then $\frac{3n}{8} < b < \frac{2n}{5} < \frac{3n}{7} \le c < \frac{n}{2}.$

If $\gcd(7, n)=1$, then let $m = 7$ and $k=3$. Since $\frac{3n}{8} < b < c < \frac{n}{2}$, $ma = m (c-b+1) \le 7 \time (\frac{n-1}{2} - \frac{3n+1}{8} +1) < n$, and we are done.

Next assume that $7 \mid n$, i.e. $n = 5 ^{\alpha} \cdot 7^{\beta}$. Note that $8 < \frac{4n}{c} \le 10 < \frac{4n}{b} < 12$.

If $9 \not\in [\frac{4n}{c}, \frac{4n}{b})$, then $\frac{4n}{c} >9$. Let $m= 12$ and $k=5$. Since $\frac{5n}{c} < 12 < \frac{5n}{b} $ and $\frac{3n}{8} < b < c < \frac{4n}{9}$,  we have $ma = m(c-b+1) \le 12 \time (\frac{4n-1}{9} - \frac{3n+1}{8} + 1) < n$, and we are done.

If $9 \in [\frac{4n}{c}, \frac{4n}{b})$, then $\frac{4n}{c}  \le 9 < 10 < \frac{4n}{b}$ and thus $ \frac{3n}{8} < b < \frac{2n}{5} < \frac{4n}{9} < c < \frac{n}{2}.$ So $$ 8n+\frac{n}{2} < \frac{69n}{8} < 23b < \frac{46n}{5} < 9n + \frac{n}{2}< 10n < \frac{92n}{9} < 23c < \frac{23n}{2} = 11n + \frac{n}{2}.$$

Note that $a = c-b+1 \le \frac{n-1}{2} - \frac{3n+1}{8} +1 = \frac{n+3}{8}$.  If $a > \frac{n}{8}$,  then let $M=12$ (note that $\gcd(n , 12)=1$). We obtain that  $|Ma|_n > \frac{n}{2}$ and $|Mb|_n > \frac{n}{2}$, and we are done. If $a < \frac{n}{8}$, since $\gcd (9, n)=1$, we may assume that $a > \frac{n}{9}$ (for otherwise, let $m = 9$ and $k=4$, we have $ma < n$, and we are done). Then $\frac{n}{9} < a < \frac{n}{8}$,  and thus $$2n+\frac{n}{2} < \frac{23n}{9} < 23 a < \frac{23n}{8} < 3n .$$


If $23c < 11n$, then $\frac{n}{9} < a = c - b + 1 \le \frac{11n-1}{23} - \frac{3n+1}{8} + 1 = \frac{19n + 57}{184}$,  which implies that $n < 40$, yielding a contradiction. So we must have  $23c > 11n$. Similarly, we can show that $23b < 9n$. Then $|23|_n + |23c|_n + |23(n-b)|_n + |23(n-a)|_n = 23 + (23c -11n) + (9n -23b) + (3n -23a) =n$ and we are done.

\medskip
\noindent {\bf Case 3.} $6 < \frac{3n}{c} \le 7  < 8 < \frac{3n}{b} < 9.$ Then $ \frac{n}{3} < b < \frac{3n}{8} < \frac{3n}{7} \le c < \frac{n}{2}$.

Note that $ a = c -b +1 \le \frac{n-1}{2} - \frac{n+1}{3} + 1 =\frac{n+1}{6}$ and $\gcd(n, 3)=1$. If $a > \frac{n}{6}$, let $M=3$. Then $|3a|_n > \frac{n}{2}$ and $|3c|_n < \frac{n}{2}$, and we are done. Next assume that $a < \frac{n}{6}$.

\noindent{\bf Subcase 3.1.} $ \gcd(7, n) = \gcd(11, n)  = 1$.

We may assume that $a > \frac{n}{7}$ (for otherwise, if  let $m = 7$ and $k=3$, we have $ma < n$, so the lemma  follows from Lemma \ref{mainlemma} (1)).  Hence $n < 11a < 2n$. Also, we have that $3n < \frac{11n}{3} < 11b < \frac{33n}{8} < 5n$ and $4n < \frac{33n}{7} < 11c < \frac{11n}{2} < 6n.$

If $ 11b < 4n$ and $11c > 5n $, we have $|11|_n + |11c|_n + |11(n-b)|_n + |11(n-a)|_n = 11 + 11c -5n + 4n -11b + 2n -11a =n$ and thus $\ind(S)=1$.

If $11b > 4n$ and $11c < 5n$, we have $|11|_n + |11c|_n + |11(n-b)|_n + |11(n-a)|_n = 11 + 11c -4n + 5n -11b + 2n -11a =3n$ and thus $\ind(S)=1$ (by Remark \ref{remark} (2)).

If $11b < 4n$ and $11c < 5n$, then $ \frac{n}{7} < a = c- b+ 1 \le \frac{5n -1}{11} - \frac{n+1}{3} +1 = \frac{4n+19}{33}$, so $n < 27$, a contradiction.

If $11b > 4n$ and $11c > 5n$, then $ \frac{n}{7} < a = c- b+ 1 \le \frac{n -1}{2} - \frac{4n+1}{11} +1 = \frac{3n+9}{22}$, so $n < 63$, again a contradiction.

\noindent {\bf Subcase 3.2.} $n = 5 ^{\alpha} \cdot 11^{\beta}$.

As in Subcase 3.1, we may assume that $a > \frac{n}{7}$. Then $ \frac{3n}{2} < \frac{13n}{7} < 13a < \frac{13n}{6} < \frac{5n}{2} < 4n < \frac{13n}{3} < 13b < \frac{39n}{8} < 5n < \frac{11n}{2} < \frac{39n}{7} < 13c < \frac{13n}{2}$.

If  $13c < 6n$, then $ \frac{n}{7} < a = c- b+ 1 \le \frac{6n -1}{13} - \frac{n+1}{3} +1 = \frac{5n+23}{39}$, so $n < 41 $, yielding a contradiction. Hence we must have that $13c > 6n$, and then $|13c|_n < \frac{n}{2}$. If $13a < 2n$ or $ 13b > \frac{9n}{2} $, then $|13a|_n > \frac{n}{2}$ or $|13b|_n > \frac{n}{2}$. Since $\gcd(n, 13)=1$,   the lemma follows from
Lemma~\ref{mainlemma} (2) with $M=13$.

Next assume that $13a > 2n$ and $ 13b < \frac{9n}{2} $. Then $\frac{2n}{13} < a < \frac{n}{6}$ and $\frac{n}{3} < b <\frac{9n}{26}$. Therefore, $$ \frac{5n}{2} < \frac{34n}{13} < 17a < \frac{17n}{6} < 3n < \frac{11n}{2} < \frac{17n}{3} < 17 b < \frac{153n}{26} < 6n.$$ We infer that $|17a|_{n} > \frac{n}{2}$ and $|17b|_{n} > \frac{n}{2}.$ Since $\gcd(n, 17)=1$,  the lemma follows from Lemma~\ref{mainlemma} (2) with $M=17$.

\noindent {\bf Subcase 3.3.} $n = 5 ^{\alpha} \cdot 7^{\beta}$.

As in Subcase 3.1, we may assume that $a > \frac{n}{8}$. By using a similar argument in Subcase 3.2,  we can complete the proof with  $M=11$ or $M=13$.
\end{proof}

\medskip
\begin{lemma}\label{k = 2}
If $k_1=2$, then $\ind(S)=1$.
\end{lemma}

\begin{proof}
Note that $2 = k_1 \le s \le 7$. By Lemma \ref{small-s} (i) we have   $ \frac{n}{b} <6$. We distinguish several cases according to the value of $\frac{n}{b}$.

\medskip
\noindent {\bf Case 1.} $5 < \frac{n}{b} < 6.$ Since $\lceil \frac{n}{c} \rceil = \lceil \frac{n}{b} \rceil$, we have that $5 < \frac{n}{c} < \frac{n}{b} < 6$. If $s\ge 4$, by Lemma \ref{small-s} (ii) we have $\frac{n}{b} < 4,$  which yields a contradiction. If $2 \le s \le 3$, applying Lemma \ref{bound of n} with $u=5$ and $v=6$, we infer that $n < 120$, yielding a contradiction again.

\medskip
\noindent {\bf Case 2.} $4 < \frac{n}{b} \le 5.$ Since $\lceil \frac{n}{c} \rceil = \lceil \frac{n}{b} \rceil$, we infer that $4 < \frac{n}{c} < \frac{n}{b} \le 5$, and then $8 < \frac{2n}{c} \le m_1 < \frac{2n}{b} \le 10$,  so $m_1=9$. Now $\gcd(n, m_1)>1$, i.e, $\gcd(9, n)>1$ yields a contradiction to $\gcd(n, 6)=1$.

\medskip
\noindent {\bf Case 3.} $3 < \frac{n}{b} < 4.$ Since $\lceil \frac{n}{c} \rceil = \lceil \frac{n}{b} \rceil$, we have $3 < \frac{n}{c} < \frac{n}{b} < 4,$ thus $6 < \frac{2n}{c} \le m_1 < \frac{2n}{b} <8.$ Hence $m_1=7$ and thus $ \frac{7}{2} < \frac{n}{b} < 4$.  Since $\gcd(n, m_1)>1$, we obtain that $ 7 \mid n$.

If $s \ge 6$,  by Lemma \ref{small-s} (iv) we have that $\frac{n}{b} < \frac{14}{5} < 3$, a contradiction.

If $4 \le s \le 5$,  by Lemma \ref{small-s} (ii), we can  assume that $[\frac{(2s-2t-1)n}{2b}, \frac{(s-t)n}{b}]$ contains only one integer for each $t \in [0, \lfloor \frac{s}{2} \rfloor-1 ]$. Since $\frac{7}{2} < \frac{n}{b} < 4$,  $12< \frac{49}{4} < \frac{7n}{2b}  < 14 < \frac{4n}{b} < 16.$ Hence $14$ is the only integer in $[\frac{7n}{2b} , \frac{4n}{b}]$, and thus   $ 13 < \frac{7n}{2b} < 14 < \frac{4n}{b} < 15.$ Then $\frac{26}{7} < \frac{n}{b} < \frac{15}{4}$. If $s=5$, then $ 16 < \frac{117}{7} < \frac{9n}{2b} < \frac{135}{8} < 17 < 18 < \frac{130}{7} < \frac{5n}{b} < \frac{75}{4} < 19$, a contradiction (since   $[\frac{9n}{2b}, \frac{5n}{b}]$ contains only one integer). If $s=4$,  then $ 9 < \frac{65}{7} < \frac{5n}{2b} < \frac{75}{8} < 10 < 11 < \frac{78}{7} < \frac{3n}{b} < \frac{45}{4} < 12$, a contradiction (since $[\frac{5n}{2b}, \frac{3n}{b}]$ contains only one integer).

Next assume that  $s=3$. Since $\frac{7}{2} < \frac{n}{b} < 4$,  $\  8 < \frac{35}{4} < \frac{5n}{2b} < 10 < \frac{21}{2} < \frac{3n}{b} < 12$ and thus $10 \in [\frac{5n}{2b}, \frac{3n}{b}]$.   By \eqref{no integer co-prime to n} we have $\gcd(10, n)>1$. Since $\gcd(n, 6)=1$, we infer that  $5 \mid n$ and $n=5^{\alpha} \cdot 7^{\beta}$. Then $ \gcd(11, n) = \gcd(9, n) = 1$. So by \eqref{no integer co-prime to n}, both $ 9$ and $ 11$ are not in $ [ \frac{5n}{2b},  \frac{3n}{b} ]$. Therefore $9 < \frac{5n}{2b} < 10 < \frac{3n}{b} < 11$ and thus $ \frac{18}{5} < \frac{n}{b} < \frac{11}{3}$. Note that $12 < \frac{4n}{c} \le 14 <  \frac{4n}{b} < 16$ and $\gcd(n, 13)=1$. If $13 \in [\frac{4n}{c}, \frac{4n}{b})$, let $m=13$ and $ k=4$. Since $3 < \frac{n}{c} <\frac{n}{b} < \frac{11}{3}$, we have $\frac{3n}{11}< b < c < \frac{n}{3}$. Then $ma = 13 \time ( c- b+1) \le 13 \time (\frac{n-1}{3}-\frac{3n+1}{11} +1 ) < n$ and we are done. Hence we may assume that $13 \not\in [\frac{4n}{c}, \frac{4n}{b})$. Then $\frac{4n}{c} > 13$ and thus $c < \frac{4n}{13}$. Since $\frac{n}{c} < \frac{10}{3} < \frac{18}{5}< \frac{n}{b}$, we obtain that $\frac{5n}{c} < 18  < \frac{5n}{b}$. Let $m = 18$ and $k = 5$. Since $\gcd(18, n)=1$ and $\frac{3n}{11} < b < c < \frac{4n}{13}$, $ma = 18 \time ( c- b+1) \le 18 \time (\frac{4n-1}{13}-\frac{3n+1}{11} +1 ) < n$, and we are done.

Finally, assume that $s=2$. Since $\frac{3n}{2b} < \frac{3}{2} \time 4 = 6 < 7< \frac{2n}{b}$, we have $6\in [\frac{3n}{2b}, \frac{2n}{b}]$. By \eqref{no integer co-prime to n} we have $\gcd(6, n) > 1$, a contradiction.

\medskip
\noindent {\bf Case 4.} $2 < \frac{n}{b} < 3.$ Since $\lceil \frac{n}{c} \rceil = \lceil \frac{n}{b} \rceil$, we have $2 < \frac{n}{c} < \frac{n}{b} <3,$ and thus $4 < \frac{2n}{c} \le m_1 < \frac{2n}{b} < 6,$ so $m_1=5$. Since $\gcd(n, m_1)>1$, we have $5 \mid n$. The result now follows from Lemma \ref{k = 2 and 5 | n}.
\end{proof}

\bigskip


Now Proposition \ref{case k > 1 and a is small}  follows immediately from Lemmas  \ref{k < 4}, \ref{k = 3} and \ref{k = 2}.

\bigskip

\section{Applications}

\medskip

In this section, we give two applications of our main result. It is shown that our main result (Theorem \ref{maintheorem1})  implies that Problem~\ref{problem} has an affirmative answer for the case when the order of $G$ is a prime power as well as the case when $|G|$ is a product of two different primes. We first remark that by using  a similar, but much simpler and shorter, argument as in the proof of Proposition \ref{maintheorem}, we obtain the following result.

\begin{proposition}\cite[Proposition 2.1]{LPYZ:10}\label{theorem2}
Let $|G|=n=p^{\alpha}$ where $p\in \mathbb P$ and $\gcd(p,6)=1$, $ \alpha \in \N$. Let $S=(g)(cg)((n-b)g)((n-a)g) $ be a minimal zero-sum sequence over $G$ such that $\ord(g)=n$, $1+c=a+b$ and $1 < a \le b < c < \frac{n}{2}.$ Then $\ind(S)=1$.
\end{proposition}

When the order of the group $G$ is a prime power (i.e. $|G|=p^k$), without loss of generality, we may assume that each minimal zero-sum sequence S can be written in the following form:
\begin{equation}\label{simplifiedform}
S=(p^lg)((p^lx_1)g)((p^lx_2)g)((p^lx_3)g), \mbox{where } \,\, 1\leq x_1, x_2, x_3 < \frac{n}{p^l}=p^{k-l} \mbox{and } |G|=\ord(g).
\end{equation}

Let $g_1=p^lg$, and then $S$ can be rewritten as $T =(g_1)(x_1g_1)(x_2g_1)(x_3g_1)$, which can regarded as a minimal zero-sum sequence over the subgroup $G_1=\langle g_1 \rangle$. The question of determining whether or not the index of $S$ (over $G$) is 1 is reduced to that of  determining whether or not the index of $T$ (over $G_1$) is 1. By applying Proposition \ref{theorem2} (and some simple observations), it is easy to show the latter is always the case, answering Problem~\ref{problem} affirmatively for the prime power case.

\medskip

\begin{theorem}[Li, Plyley, Yuan and Zeng (2010)]\label{theorem for prime power}
Let $G$ be a cyclic group of prime power order such that $\gcd(|G|, 6)=1$. Then every minimal zero-sum sequence $S$ over $G$ of length $|S|=4$ has $\ind(S)=1$.
\end{theorem}

We note that when the order of $G$ is not necessarily a prime power (say, for example, $|G|=p^{\alpha}q^{\beta}$ is a product of two prime powers), the above mentioned reduction is not always possible as shown in the following example.

\begin{example}\label{REDUCTION-IMPOSSIBLE}

Let $|G|=n=(5^{2})(7)=175$  and  $S= (5g)(135g)(77g)(133g)$ where $\ord(g)=n$. Clearly, 
$S$ cannot be reduced to the simplified form as described in (\ref{simplifiedform}). Note that $|4\times 5|_n+|4\times 135|_n+|4 \times 77|_n
+|4 \times 133|_n=20+15+133+7=175=n$, so  we obtain that $\ind(S)=1$.

\end{example}

The obstruction mentioned in the above example prevents us to apply directly our main result to answer the problem affirmatively for the general case when the order of $G$ is a product of two prime powers. However, our main result does apply to several new special situations. We conclude the paper by providing one such case.

\begin{theorem}\label{thm-pq}
If $|G|=pq$ with $p\ne q$ primes and $\gcd(pq,6)=1$, then Problem~\ref{problem}  has an affirmative answer.
\end{theorem}

\begin{proof}
Let $S = (x_1 g)(x_2g)(x_3g)(x_4g)$  be a minimal zero-sum sequence over $G$ such that
 $\mbox{ord}(g)=n$.  By Theorem~\ref{maintheorem1}, we may assume that $\gcd(n, x_i) > 1$ for every $i \in [1,4]$.

If $\gcd(n, x_i) =p$ for every $i\in [1,4]$, let $g_1=pg$. As mentioned above   $S$ can be rewritten as $T =(\frac{x_1}{p}g_1)(\frac{x_2}{p}g_1)(\frac{x_3}{p}g_1)(\frac{x_4}{p}g_1)$, which can be regarded as a minimal zero-sum sequence over the subgroup $G_1=\langle g_1 \rangle$, where $|G_1|=\ord(g_1)=q$. By Theorem \ref{theorem for prime power}, we have $\ind(T)=1$ in $G_1$, i.e. there exists $m \in [1, q-1]$  such that $\gcd(m , q)=1$ and  $ |\frac{mx_1}{p}|_q+ \, \cdots \, +|\frac{mx_4}{p}|_q= q$. Then $ |mx_1|_{pq}+\, \cdots \, +|mx_4|_{pq} = |(m+q)x_1|_{pq}+\, \cdots \, +|(m+q)x_4|_{pq}= pq  \, (*)$.  Note that either $\gcd(m, pq)=1$ or $\gcd(m+q, pq)=1$.  Now $(*)$ implies that $\ind(S)=1$.

If $\gcd(n, x_i) =q$ for every $i\in [1,4]$, let $g_2=qg$. As above we  obtain that $\ind(S)=1$.

Next without loss of generality, we may assume that $\gcd(n, x_1)=\gcd(n, x_2)=p$ and $\gcd(n, x_3)=\gcd(n, x_4)=q$. Since $x_1+x_2+x_3+x_4 \equiv 0\pmod {pq}$, we must have $q \mid x_1+x_2$. Then $pq \mid x_1+x_2,$ yielding a contradiction to the assumption that $S$ is minimal zero-sum sequence. This completes the proof.
\end{proof}


\bigskip





\begin{thebibliography}{10}



\bibitem{CFS:99} S.T. Chapman, M. Freeze, and W.W Smith, \emph{Minimal zero sequences and the strong Davenport constant}, Discrete Math. \textbf{203}(1999), 271-277.

\bibitem{CS:05} S.T. Chapman,  and W.W Smith, \emph{A characterization of minimal zero-sequences of index one in finite cyclic groups}, Integers \textbf{5(1)}(2005), Paper A27, 5p.

\bibitem{Gao:00}
W. Gao ,  \emph{Zero sums in finite cyclic groups}, Integers
\textbf{0} (2000), Paper A14, 9p.


\bibitem{GG:09}
W. Gao and A. Geroldinger, \emph{On products of $k$ atoms}, Monatsh. Math. \textbf{156} (2009), 141-157.

\bibitem{GLPPW:11}
W. Gao, Y. Li, J. Peng, P. Plyley and G. Wang \emph{On the index of sequences over cyclic groups} (English), Acta Arith. \textbf{148}, No. 2, (2011) 119-134 .


\bibitem{G:87}
 A. Geroldinger, \emph{On non-unique factorizations into irreducible elements. II}, Number Theory, Vol II Budapest 1987, Colloquia Mathematica Societatis Janos Bolyai, vol. 51, North Holland, 1990, 723-757.

 \bibitem{Ge:09a}
A.~Geroldinger, \emph{Additive group theory and non-unique
factorizations},
  Combinatorial {N}umber {T}heory and {A}dditive {G}roup {T}heory
  (A.~Geroldinger and I.~Ruzsa, eds.), Advanced Courses in Mathematics CRM
  Barcelona, Birkh{\"a}user, 2009, pp.~1 -- 86.

\bibitem{Ge-HK06a}
A.~Geroldinger and F.~Halter-Koch, \emph{Non-{U}nique
{F}actorizations.  {A}lgebraic, {C}ombinatorial and {A}nalytic
{T}heory}, Pure and Applied  Mathematics, Vol. 278, Chapman \&
Hall/CRC, 2006.

\bibitem{KL:89} D. Kleitman and P. Lemke, \emph{An addition theorem on the integers modulo $n$}, J. Number Theory \textbf{31}(1989), 335-345.

\bibitem{LPYZ:10}
Y.~Li, C. Plyley, P.~Yuan and X.~Zeng, \emph{Minimal zero sum sequences of length four over finite cyclic groups}, Journal of Number Theory.  \textbf{130} (2010), 2033 -- 2048.

\bibitem{P:04}
V. Ponomarenko, \emph{Minimal zero sequences of finite cyclic
groups}, Integers  \textbf{4}(2004), Paper A24, 6p.
\bibitem{SC:07} S. Savchev and F. Chen, \emph{Long zero-free sequences in finite cyclic groups}, Discrete Math. \textbf{307}(2007), 2671-2679.

\bibitem{XY:09} X. Xia and P. Yuan, \emph{Indexes of insplitable minimal zero-sum sequences of length $l(C_n)-1$}, Discrete Math.
\text{310}(2010)£¬1127-1133.

\bibitem{Y:09} P. Yuan, \emph{On the index of minimal zero-sum sequences over finite cyclic
groups}, J. Combin. Theory Ser. A \textbf{114}(2007), 1545-1551.

\end{thebibliography}
\end{document}